\documentclass[11pt, a4paper, reqno]{amsart}
\usepackage{amsfonts,amsmath,amsthm,amssymb,mathrsfs}
\usepackage{color}
\usepackage{amsmath,amssymb,amsfonts,bm}
\usepackage{graphicx}
\usepackage{newunicodechar}
\usepackage{xcolor}
\numberwithin{equation}{section}
\usepackage[pdfstartview=FitH,colorlinks=true]{hyperref}
\hypersetup{urlcolor=cyan, linkcolor=blue,citecolor=red}
\allowdisplaybreaks

\theoremstyle{plain}

\newtheorem{thm}{Theorem}[section]
\newtheorem{definition}[thm]{Definition}
\newtheorem{prop}[thm]{Proposition}
\newtheorem{cor}[thm]{Corollary}
\newtheorem{lemma}[thm]{Lemma}
\newtheorem{remark}[thm]{Remark}

\DeclareMathOperator{\rank}{rank}
\makeatletter

\newcommand{\Rmnum}[1]{\expandafter\@slowromancap\romannumeral #1@}
\makeatother

\begin{document}
\title[Degenerate parabolic equation]
{Analytic smoothing effect of the Cauchy problem for a class of ultra-parabolic equations}

\author[X.-D.Cao \& C.-J. Xu]
{Xiao-Dong Cao and Chao-Jiang Xu}

\address{Xiao-Dong Cao and Chao-Jiang Xu
\newline\indent
School of Mathematics and Key Laboratory of Mathematical MIIT,
\newline\indent
Nanjing University of Aeronautics and Astronautics, Nanjing 210016, China
}
\email{caoxiaodong@nuaa.edu.cn; xuchaojiang@nuaa.edu.cn}

\date{\today}

\subjclass[]{}


\begin{abstract}
In this paper, we study a class of strongly degenerate ultra-parabolic equations with analytic coefficients. We demonstrate that the Cauchy problem exhibits an analytic smoothing effect. This means that, with an initial datum belonging to the Sobolev space $H^s$  (of  real index $s$), the associated Cauchy problem admits a unique solution that is analytic in all spatial variables for any strictly positive time. This smoothing effect property is similar to that of the Cauchy problem for uniformly parabolic equations with analytic coefficients.
\end{abstract}

\maketitle

\section{Introduction}
In this paper, we study the following Cauchy problem
\begin{equation}\label{1-1}
\begin{cases}
\partial_t u + X u + Y u- \displaystyle \sum_{k, j = 1}^{m_0}  a_{kj}(t, x) \partial_{x_k} \partial_{x_j} u = g(t, x),\quad t>0,\ \ x\in \mathbb{R}^{n},  \\
u\big|_{t = 0} = u_0\in H^s(\mathbb{R}^{n}),
\end{cases}
\end{equation}
where $ 1 \le m_0 < n, s\in\mathbb{R}$, $X =\displaystyle \sum_{k, j = 1}^n b_{kj} x_k \partial_{x_j} $ is a vector fields with real constant coefficients $(b_{k j})$, $Y =\displaystyle \sum^{m_0}_{\ell=1}b_\ell(t, x)\partial_{x_\ell}+b_0(t, x)$ is a first-order partial differential operator with real smooth coefficients and $(a_{kj}(t, x))$ is a symmetric $m_0\times m_0$ matrix  with real smooth coefficients.

The interest of studying this type of strongly degenerate ulta-parabolic equation arises from many applications, both in physics \cite{ILL-2}\cite{ILL-1}\cite{Kol-1}\cite{Villani-1}; finance \cite{Bar-1}\cite{BS}\cite{fi-1} and the mathematical contexts \cite{Bramanti}\cite{ref18}\cite{w-1}. For detailed motivations, we refer to Section \ref{S-7} of this paper. 

The purpose of this work is to investigate the analytic smoothing effect of the Cauchy problem \eqref{1-1}, where the coefficients are analytic. Specifically, for an initial datum $u_0\in H^s(\mathbb{R}^{n})$ with $s \in \mathbb{R}$, particularly for $s\le 0$, the solution of the Cauchy problem \eqref{1-1} is analytic for all spatial variables, i. e. $u(t)\in\mathcal{A}(\mathbb{R}^{n})$ for any $t>0$. It is well known that the main challenge in studying the regularity of this type of equation lies in overcoming its high degeneracy. To address this problem, a key aspect of our work is the construction of a family of well-chosen, time-dependent auxiliary vector fields. For the specific form of these vector fields, we refer to Section \ref{sectionsec}.

Now we provide the definition of the function space. Let $\Omega\subset\mathbb R^{n}$ be an open domain, the analytic function space $\mathcal A(\Omega)$ consists of $C^{\infty}$ functions $f$ that satisfy the following condition: there exists a constant $C>0$ such that,  for any $\alpha\in\mathbb N^{n}$,
$$
\left\| \partial^\alpha f \right\|_{L^\infty(\Omega)} \leq C^{\left| \alpha \right| + 1} \alpha !.
$$
By using the Sobolev embedding, we can replace the $L^{\infty}$ norm by the $L^{2}$ norm, or the norm in any Sobolev space in the above definition.

To state the main result, we first introduce the following notations. For any integer $q \ge 1$, define
\begin{equation}\label{initial vector fields}
\begin{aligned}
&{\bf X}_{0} =(\partial_{x_1}, \cdots, \partial_{x_{m_0}})=(X_{1, 0}, \cdots, X_{m_0, 0}),\\
&{\bf X}_{1} = \left[ {\bf X}_0, X \right] = ( \left[ \partial_{x_1}, X \right],\cdots,  \left[ \partial_{x_{m_0}}, X \right])=(X_{1, 1}, \cdots, X_{m_0, 1}),\\
&{\bf X}_{q} = \left[ {\bf X}_{q - 1}, X \right] =( \left[ X_{1, q-1}, X \right],\cdots,  \left[ X_{m_0, q-1}, X \right])=(X_{1, q}, \cdots, X_{m_0, q}),
\end{aligned}
\end{equation}
here $\left[ \cdot, \ \cdot \right]$ denotes the commutator. Then $\left\{ X_{p, q}; \ p = 1, \dots, m_0, \ q \in \mathbb{N} \right\}$ are vector fields with real constant coefficients.

Meanwhile, we give the following assumptions:
\begin{itemize}
  \item [{$\bf (H_1)$}] H\"ormander's condition:
$$
\begin{cases}
\mbox{There exists $r \ge 1$ such that ${\bf X}_{q} = 0$ if $q > r$,  ${\bf X}_{r} \not= 0$ and $\left\{ {\bf X}_0, {\bf X}_1, \dots, {\bf X}_r \right\}$}\\
\ \mbox{span the tangent space $T\mathbb{R}^n$.}
\end{cases}
$$
\item [{$\bf (H_2)$}] Partial diffusion:
Let  $T>0$,  there exists a constant $\Lambda > 0$ such that, for all $(t, x)
\in [0, T]\times\mathbb{R}^{n},\ \xi\in\mathbb{R}^{m_0}$,
\begin{equation*}%\label{1-2+}
 \Lambda^{-1} \sum_{k = 1}^{m_0} \xi_k^2 \leq \sum_{k, j = 1}^{m_0} a_{kj}(t, x) \xi_k \xi_j
 \leq \Lambda \sum_{k = 1}^{m_0} \xi_k^2\  .
\end{equation*}
\item [{$\bf (H_3)$}] Analyticity of coefficients: Let $T>0$, there exists constants $B, \ C>0$ such that for any $\alpha \in\mathbb N^n$, $k, j \in\{1, \cdots, m_0\}$,
    $\ell \in \{ 0, \dots, m_0 \}$ and $ 0 \le t \leq T$, we have $a_{k j}, b_\ell, g\in C^{\infty}([0, T]\times\mathbb{R}^n )$ and  
\begin{equation}\label{1-2}
\begin{aligned}
&\left\| \partial^{\alpha} a_{kj}(t) \right\|_{L^2 \left( \mathbb{R}^n \right)} \leq B^{|\alpha| + 1} \alpha!, \quad
\left\| \partial^{\alpha} b_{\ell}(t) \right\|_{L^2 \left( \mathbb{R}^n \right)}\leq B^{|\alpha| + 1} \alpha!,\\
&\left\| \partial^{\alpha} g(t) \right\|_{L^2 \left( \mathbb{R}^n \right)} \leq C^{|\alpha| + 1} \alpha!. 
\end{aligned}
\end{equation}
\end{itemize}
The main result of this paper can be stated as follows.

\begin{thm}\label{thm1}
Let $T > 0, s\in\mathbb{R}$ and $u_0 \in H^s(\mathbb{R}^n)$. Assume that the equation in \eqref{1-1} satisfies the assumptions {$\bf (H_1), (H_2), (H_3)$}.
Then the Cauchy problem \eqref{1-1} admits a unique solution $u \in C^\infty ( ]0, T], \mathcal{A}(\mathbb{R}^n))$. Moreover, for $\delta > 1$, there exists a constant $L > 0$, such that for any $\alpha \in \mathbb{N}^n$,
\begin{equation}\label{analy-11}
\sup_{0 < t \leq T} t^{(\delta + 2r)|\alpha|} \| \partial^\alpha u(t) \|_{H^s(\mathbb{R}^n)} \leq L^{\left| \alpha \right| + 1} \alpha!.
\end{equation}
\end{thm}

\begin{remark}\label{remark1}
The H\"ormander's  hypoelliptic theorem of \cite{hormander1} implies that any weak solution of \eqref{1-1} is locally $C^\infty$ for all variables. Additionally, \cite{derridj-zuily} improves this by showing that the weak solution locally belongs to the Gevrey class of index greater than or equal to $2r+1$. The new observation of this work is that any weak solution of the Cauchy problem \eqref{1-1} is, in fact, analytic for all spatial variables.

It is well know that the uniformly parabolic equation exhibits the analytic smoothing effect property of Cauchy problem, so we extend this property to a very general and strongly degenerate parabolic equation.

A typical example of Cauchy problem for ultra-parabolic equation is stated as follows 
$$
\begin{cases}
\partial_t u + \displaystyle\sum^{n-1}_{j=1}x_j\partial_{x_{j+1}} u - a(t, x) \partial^2_{x_1} u = g,\quad t>0,\ \ x\in \mathbb{R}^{n},  \\
u\big|_{t = 0} = u_0\in H^s(\mathbb{R}^{n}),
\end{cases}
$$
where $n\ge 2$ and $s\in\mathbb{R}$, in particular, $s$ can be any negative index. Then the above Cauchy problem  admits  a  weak solution $u(t, x)$ which is analytic for $x\in\mathbb{R}^n$ when $t>0$, if $a$ and $g$ satisfy {$\bf (H_2)$} and {$\bf (H_3)$}.

We will give more examples with different motivations and backgrounds in Section \ref{S-7}.
\end{remark}

The rest of this paper is organized as follows. In Section \ref{section2}, we provide a proof of the existence theorem for the smoothing solution with smooth initial datum. In Section \ref{sectionsec}, we construct a system of time-dependent auxiliary vector fields, which is the core part of this paper, to help overcome the high degeneracy in the spatial variables. Additionally, several commutator estimates are provided. In Section \ref{section5}, we present \`{a} priori estimates for the auxiliary vector fields. In Section \ref{section6}, we prove the main result of Theorem \ref{thm1}. In Section \ref{S-6}, we offer supplementary proofs for technical lemmas that are directly utilized in Sections \ref{section2} and \ref{section6}. Finally, in Section \ref{S-7}, we present several specific examples for the target equation in the Cauchy problem \eqref{1-1}, which illustrate the motivations and background of our topic.

\section{Existence of the smooth solution}\label{section2}
In this section, we prove the existence of the smooth solution with smooth initial datum. Remark that \cite{Des-Vil}, 
$$
\mathscr{S}(\mathbb{R}^n) =\bigcap_{m, \mu \in \mathbb{N}} H^m_\mu (\mathbb{R}^n),
$$
and the norm in $H^m_\mu (\mathbb{R}^n)$ can be stated as follows \cite{Equal},    
$$
\|v\|^2_{H^m_\mu(\mathbb{R}^n)} = \left\| \left< x \right>^\mu \Lambda_x^m v \right\|_{L^2(\mathbb{R}^n)}^2 \approx \left\| \Lambda_x^m (\left< x \right>^\mu v) \right\|_{L^2(\mathbb{R}^n)}^2.
$$
where $\left< x \right>=(1+|x|^2)^{\frac 12}$ and $\Lambda_x = (1-\triangle_x)^{\frac 12}$, i. e. the Fourier multiplier of the symbol $\left< \xi \right>$ with respect to $x \in \mathbb{R}^n$.

First, we establish the following commutator estimate, which is used throughout the paper. However, to maintain the readability of the paper, the proof of this estimate is postponed to Section \ref{S-6}.
\begin{lemma}\label{lemma-0}
Let $s \in \mathbb{R}$, there exists a constant $C_1 > 0$ which is dependent on $s, n$, such that  for any suitable functions $h $ and $f $, we have that, 

(1) for the commutators,  
\begin{equation}\label{commu-1}
\left\| \left[ h, \Lambda_x^s \right] f \right\|_{L^2(\mathbb{R}^n)} \leq C_1 \left\| h \right\|_{H^{s_0}(\mathbb{R}^n)} \left\|  f \right\|_{H^{s - 1}(\mathbb{R}^n)}, 
\end{equation}
where $s_0 =  |s - 1|  + \frac{n}{2} + 2$ .

(2) for the multiplication, 
\begin{equation}\label{commu-2}
\left\| h f \right\|_{H^s(\mathbb{R}^n)} \leq C_1 \left\| h \right\|_{H^{s_1}(\mathbb{R}^n)} \left\| f \right\|_{H^s(\mathbb{R}^n)}, 
\end{equation}
where $s_1 = |s| + \frac{n}{2} + 1$ .
\end{lemma}

Next, we present some preliminary lemmas that are used in the proof of the main result of this section.
\begin{lemma}\label{prelemma-1}
For any $m, \mu \in \mathbb{N}$, there exists a constant $C_2 > 0$ such that for any suitable functions  $f$,
\begin{equation*}
\big| \left(Xf, f \right)_{H^m_\mu(\mathbb{R}^n)} \big|+\big| \left(Y f, f \right)_{H^m_\mu(\mathbb{R}^n)} \big|  \leq C_2 \left\| f \right\|_{H^m_\mu(\mathbb{R}^n)}^2.
\end{equation*}
\end{lemma}
Note the coefficients of vector fields $X = \displaystyle \sum_{k, j = 1}^{m_0} b_{kj} x_k \partial_{x_j}$ and $Y = \displaystyle \sum_{\ell = 1}^{m_0} b_\ell (t, x) \partial_{x_\ell} + b_0(t, x)$ satisfy the assumption ${\bf{(H_3)}}$. The proof of the above lemma involves using integration by parts for the case $m=\mu=0$. For the other cases, we need to estimate the commutators $[x_k, \left< x \right>^\mu \Lambda_x^m]$, which can be computed directly, and $[b_\ell, \left< x \right>^\mu\Lambda_x^m]$, for which we can use \eqref{commu-1}. 

A similar computation also gives the following lemma. 
\begin{lemma}\label{prelemma-2}
For any $m, \mu \in \mathbb{N}$,  there exists a constant $C_3 > 0$ such that for any suitable function $f $, 
\begin{equation*}
\big| \big( \sum_{j = 1}^{m_0} \big( \sum_{k= 1}^{m_0}\partial_{x_k} a_{k j} \big) \big( \partial_{x_j} f \Big), f \big)_{H^m_\mu(\mathbb{R}^n)} \big| \leq C_3 \left\| f \right\|_{H^m_\mu(\mathbb{R}^n)}^2.
\end{equation*}
\end{lemma}
By using the lemma \ref{prelemma-2}, we also have the following result.
\begin{lemma}\label{prelemma-2+}
For any $m, \mu \in \mathbb{N}$, there exists a constant $C_4 > 0$ such that for any suitable function $f$,
\begin{equation*}
(2\Lambda)^{-1} \sum_{k = 1}^{m_0}\|\partial_{x_k} f\|^2_{H^m_\mu(\mathbb{R}^n)}\le
-\Big( \sum_{k, j = 1}^{m_0} a_{kj} \partial_{x_k}\partial_{x_j} f, f \Big)_{H^m_\mu(\mathbb{R}^n)} + C_4 \left\| f \right\|_{H^m_\mu(\mathbb{R}^n)}^2.
\end{equation*}
\end{lemma}
Here we use
\begin{align*}
-&\Big( \sum_{k, j = 1}^{m_0} a_{kj} \partial_{x_k}\partial_{x_j} f, f \Big)_{H^m_\mu(\mathbb{R}^n)}\\
&=\Big( \sum_{k, j = 1}^{m_0} \left( \partial_{x_k} a_{kj} \right) \left( \partial_{x_j} f \right), f \Big)_{H^m_\mu(\mathbb{R}^n)} - \Big( \sum_{k, j = 1}^{m_0} \partial_{x_k} \left( a_{kj} \partial_{x_j} f \right), f \Big)_{H^m_\mu(\mathbb{R}^n)}
\end{align*}
and the computation of the commutator $[a_{kj}, \left< x \right>^\mu \Lambda_x^m]$.

Now our main result of this section is stated as follows.
\begin{thm}\label{thm2}
Let $u_0 \in \mathscr{S}(\mathbb{R}^n)$, $s \in \mathbb{R}$ and $T > 0$. Assume $(a_{kj}), (b_{\ell}), g$ satisfy the assumptions {$\bf (H_1), (H_2), (H_3)$}, and $g\in C^\infty ([0, T]; \mathscr{S}(\mathbb{R}^n))$. Then the Cauchy problem \eqref{1-1} admits a unique solution $u \in C^\infty \left( [0, T], \mathscr{S}(\mathbb{R}^n) \right)$. Moreover, there exists a constant $C_T > 0$ such that for any $t \in [0, T]$,
\begin{equation}\label{energy}
\begin{aligned}
&\left\| u(t) \right\|^2_{H^s(\mathbb{R}^n)} + \Lambda^{-1} \sum_{k = 1}^{m_0} \int_0^t \left\| \partial_{x_k} u(\tau) \right\|^2_{H^s(\mathbb{R}^n)} d\tau\\
&\leq C_T\Big( \left\| u_0 \right\|^2_{H^s(\mathbb{R}^n)} +  \left\| g \right\|^2_{L^1 \left( [0, T], H^s(\mathbb{R}^n) \right)}\Big).
\end{aligned}
\end{equation}
\end{thm}
\begin{proof}
For the initial datum $u_0\in \mathscr{S}(\mathbb{R}^n)$,  we first prove the existence of solution
$$
u \in L^\infty([0, T];H^{m}_\mu (\mathbb{R}^n) )
$$
for the Cauchy problem \eqref{1-1} for any $m, \mu \in \mathbb{N}$.

Let
$$
P = -\partial_t + \big( X + Y - \sum_{k, j = 1}^{m_0} a_{kj}  \partial_{x_k} \partial_{x_j} \big)^{*},
$$
here the adjoint operator $\big(\ \cdot\ \big)^{*}$ is taken with respect to the scalar product in the space $H^m_{\mu}(\mathbb{R}^n)$. For any $0 \leq t \leq T$ and $f \in C^\infty ([0, T], \mathscr{S}(\mathbb{R}^n))$ with $\partial^\alpha f(T) = 0$ for any $\alpha \in \mathbb{N}^n$, we have
$$
\left( f, Pf \right)_{H^m_\mu} = -\frac{1}{2} \frac{d}{dt} \left\| f(t) \right\|^2_{H^m_\mu} + \left( (X + Y) f, f \right)_{H^m_\mu}  - \Big( \sum_{k, j = 1}^{m_0}  a_{kj}\partial_{x_k} \partial_{x_j} f, f \Big)_{H^m_\mu}.
$$
By using the Lemma \ref{prelemma-1} and Lemma \ref{prelemma-2+}, we can obtain the following inequality 
$$
- \frac{d}{dt} \left\| f(t) \right\|^2_{H^m_\mu} + \Lambda^{-1} \sum_{k = 1}^{m_0} \left\| \partial_{x_k} f \right\|_{H^m_\mu}^2 \le C_0 \left\| f \right\|^2_{H^m_\mu} + 2 \left\| f \right\|_{H^m_\mu} \left\| Pf \right\|_{H^m_\mu},
$$
where $C_0 = 2(C_2 + C_4)$.
Since $\partial^\alpha f(T) = 0$ for all $\alpha \in \mathbb{N}^n$, we have 
\begin{equation*}
\begin{aligned}
& \left\| f(t) \right\|^2_{H^m_\mu} + \Lambda^{-1} \sum_{k = 1}^{m_0} \int_t^T e^{C_0(\tau - t)} \left\| \partial_{x_k} f(\tau) \right\|^2_{H^m_\mu} d\tau \\
&\qquad \leq 2e^{C_0 T} \left\| f \right\|_{L^\infty \left( [0, T], H^m_\mu (\mathbb{R}^n) \right)} \left\| Pf \right\|_{L^1 \left( [0, T], H^m_\mu (\mathbb{R}^n) \right)},
\end{aligned}
\end{equation*}
which leads to
\begin{equation}\label{sec3}
\left\| f \right\|_{L^\infty \left( [0, T], H^m_\mu (\mathbb{R}^n) \right)} \leq 2e^{C_0 T} \left\| Pf \right\|_{L^1 \left( [0, T], H^m_\mu (\mathbb{R}^n) \right)}.
\end{equation}
For $m, \mu \in \mathbb{N}$, we consider the subspace of $L^1 \left( [0, T], H^m_\mu (\mathbb{R}^n) \right)$,
$$
\mathcal{F} = \left\{ Pf; f \in C^\infty \left( [0, T], \mathscr{S}(\mathbb{R}^n) \right), \partial^\alpha f(T) = 0, \forall \alpha \in \mathbb{N}^n \right\} .
$$
Define the linear functional
\begin{equation*}
\begin{aligned}
\mathcal{G} :\ \ \mathcal{F}\ \ \ &\rightarrow \ \mathbb{R} \\
\omega=Pf\ &\mapsto \ \left( u_0, f(0) \right)_{H^m_\mu (\mathbb{R}^n)} + \int_0^T \left( g(\tau), f(\tau) \right)_{H^m_\mu (\mathbb{R}^n)} d\tau.
\end{aligned}
\end{equation*}
Let $Pf_1 = Pf_2$ with $f_1, f_2 \in C^\infty \left( \left[ 0, T \right], \mathscr{S}(\mathbb{R}^n) \right)$ and $f_1 (T) = f_2 (T) = 0$. Then from \eqref{sec3}, it follows that
$$
\left\| f_1 - f_2 \right\|_{L^\infty \left( \left[0, T \right], H^m_\mu (\mathbb{R}^n) \right)} \leq 2e^{C_0 T} \left\| Pf_1 - Pf_2 \right\|_{L^1 \left( [0, T], H^m_\mu (\mathbb{R}^n) \right)} = 0,
$$
hence $f_1 = f_2$, which leads to the fact that the operator $P$ is injective. Therefore, the linear functional $\mathcal{G}$ is well-defined. Moreover, we can obtain
\begin{equation*}
\begin{aligned}
\left| \mathcal{G}\left( Pf \right) \right|
&\leq \big( \left\| u_0 \right\|_{H^m_\mu (\mathbb{R}^n)} + \int_0^T \left\| g(\tau) \right\|_{H^m_\mu (\mathbb{R}^n)} d\tau \big) \ \left\| f \right\|_{L^\infty \left( [0, T], H^m_\mu (\mathbb{R}^n) \right)} \\
&\leq 2e^{c_0 T} \big( \left\| u_0 \right\|_{H^m_\mu (\mathbb{R}^n)} + \int_0^T \left\| g(\tau) \right\|_{H^m_\mu (\mathbb{R}^n)} d\tau \big) \ \left\| Pf \right\|_{L^1 \left( [0, T], H^m_\mu (\mathbb{R}^n) \right)}.
\end{aligned}
\end{equation*}
Hence $\mathcal{G}$ is a continuous linear functional on $\big( \mathcal{F}, \left\| \cdot \right\|_{L^1 \left( [0, T], H^m_\mu (\mathbb{R}^n) \right)} \big)$. By using the Hahn-Banach theorem, $\mathcal{G}$ can be extended to a continuous linear functional on $L^1 \left( [0, T], H^m_\mu (\mathbb{R}^n) \right)$ with its norm can be bounded by
$$
2e^{C_0 T} \Big( \left\| u_0 \right\|_{H^m_\mu (\mathbb{R}^n)} + \int_0^T \left\| g(\tau) \right\|_{H^m_\mu (\mathbb{R}^n)} d\tau \Big).
$$
It follows the Riesz representation theorem that there exists a unique function $u \in L^\infty \left( [0, T], H^m_\mu (\mathbb{R}^n) \right)$ satisfying
$$
\mathcal{G}(\omega) = \int_0^T \left( u(\tau), \omega(\tau) \right)_{H^m_\mu (\mathbb{R}^n)} d\tau,\ \ \  \forall \omega \in \mathcal{F},
$$
and
$$
\left\| u \right\|_{L^\infty \left( [0, T], H^m_\mu (\mathbb{R}^n) \right)} \leq 2e^{C_0 T} \left( \left\| u_0 \right\|_{H^m_\mu (\mathbb{R}^n)} + \left\| g \right\|_{L^1 \left( [0, T], H^m_\mu (\mathbb{R}^n) \right)} \right),
$$
these imply that for all $f \in C^\infty \left( [0, T], \mathscr{S}(\mathbb{R}^n) \right)$,
\begin{equation*}
\begin{aligned}
\mathcal{G} \left( Pf \right) &= \int_0^T \left( u(\tau), (Pf)(\tau) \right)_{H^m_\mu (\mathbb{R}^n)} d\tau \\
&= \left( u_0, f(0) \right)_{H^m_\mu (\mathbb{R}^n)} + \int_0^T \left( g(\tau), f(\tau) \right)_{H^m_\mu (\mathbb{R}^n)} d\tau.
\end{aligned}
\end{equation*}
Therefore, for any $ m, \mu \in \mathbb{N}$, $u \in L^\infty \left( [0, T], H^m_\mu (\mathbb{R}^n) \right)$ is a unique solution of the Cauchy problem \eqref{1-1}. We have thus proved the Cauchy problem \eqref{1-1} admits a unique solution belonging to
$ L^\infty \left( [0, T], \mathscr{S}(\mathbb{R}^n) \right)$.

The regularity with respect to the $t$ variable can be obtained directly by using the regularity of the initial datum. Thus, we have
$$
u \in C^\infty \left( [0, T], \mathscr{S}(\mathbb{R}^n) \right).
$$

Now for $u \in C^\infty \left( [0, T], \mathscr{S}(\mathbb{R}^n) \right)$, by using the energy method, the solution of the Cauchy problem \eqref{1-1} satisfies the following inequality,  
$$
\frac{d}{dt} \left\| u(t) \right\|^2_{H^s} + \Lambda^{-1} \sum_{k = 1}^{m_0} \left\| \partial_{x_k} u \right\|^2_{H^s} \leq \left\| g (t) \right\|^2_{H^s} + C_5 \left\| u(t) \right\|^2_{H^s},
$$
here we can use similar computations as in Lemmas \ref{prelemma-1}, \ref{prelemma-2} and \ref{prelemma-2+}. Thus we have
$$
\left\| u(t) \right\|^2_{H^s} + \Lambda^{-1} \sum_{k = 1}^{m_0} \int_0^t \left\| \partial_{x_k} u(s) \right\|^2_{H^s} ds\leq C_T \left( \left\| u_0 \right\|_{H^s}^2 + \left\| g \right\|^2_{L^1(\left[0, T \right], H^s(\mathbb{R}^n))} \right) .
$$
\end{proof}

By combining the above result with compactness and the uniqueness of the solution, we can also establish the existence of a weak solution for the initial datum belonging to $H^s$.
\begin{cor}\label{weak}
Let $T > 0, s \in \mathbb{R}$ and $u_0 \in H^s(\mathbb{R}^n)$. Assume $(a_{kj}), (b_{\ell}), g$ satisfy the assumptions {$\bf (H_1), (H_2), (H_3)$}.
Then the Cauchy problem \eqref{1-1} admits a unique solution $u \in L^\infty \left( [0, T], H^s(\mathbb{R}^n) \right)$ which satisfies \eqref{energy}.
\end{cor}

\section{Time-dependent auxiliary vector fields}\label{sectionsec}
In this section, we introduce a family of well-chosen time-dependent auxiliary vector fields and
study their commutators with the partial differential operators of \eqref{1-1}.

We now introduce antiderivative operator. Let $f(t)$ a continuous function on $[0, T]$, we define
$$
I(f)(t)=\int^t_0 f(s) ds,\quad t\in [0, T],
$$
and for $k\in\mathbb{N}$, $k$-time antiderivative operators of $f$ is defined by the following iteration: 
$$
I^0(f)(t)=f(t),\ \ I^k(f)(t)=I(I^{k-1}(f))(t),\quad t\in [0, T],
$$
then $\frac{d}{dt}I^k(f)(t)=I^{k-1}(f)(t)$ for $k\ge 1$. For some $\delta>1$, setting
$$
h_\delta(t)=t^\delta,\quad t\in [0, T],
$$
then
$$
I^q(h_\delta)(t)=\frac{\Gamma(\delta + 1)}{\Gamma(\delta + 1 + q)} t^{\delta + q},
$$
where Gamma function is
$$
\Gamma (z) = \int_0^{+\infty} s^{z - 1} e^{-s} ds, \ \ z > 0,
$$
with the property
$$
z\ \Gamma(z)=\Gamma(z+1).
$$
Now we introduce the auxiliary vector fields as follows.
\begin{definition}\label{deff}
For $p \in \left\{ 1, \dots, m_0 \right\}$ and $\delta>1$, we define
\begin{equation}\label{2-1}
H_{\delta, p} = \sum_{q = 0}^r I^q(h_\delta)(t)\ X_{p, q},
\end{equation}
where $\{X_{p, q}\}$ are the vector fields defined in \eqref{initial vector fields}. So
$\{H_{\delta, p}\}$ are vector fields with coefficients depends on $t\in [0, T]$.
\end{definition}
The significant property of the above auxiliary vector fields is stated as follows:
\begin{lemma}\label{lemma1}
For $p \in \left\{ 1, \dots, m_0 \right\}$, and any $d \in \mathbb{N}^+$, we have
\begin{equation}\label{2-3}
\left[ \partial_t + X, H^d_{\delta, p} \right] = d \delta t^{\delta - 1}\partial_{x_p} H^{d - 1}_{\delta, p}   .
\end{equation}
\end{lemma}
Remark that here we need the condition $\delta>1$ to ensure the continuity of $h'_\delta(t)=\delta t^{\delta - 1}$ on $[0, T]$ .
\begin{proof}
We would prove \eqref{2-3} by induction on the index $d$. For the case of $d = 1$, since
$$
\left[ \partial_t + X, H_{\delta, p} \right] = \left[ \partial_t, H_{\delta, p} \right] + \left[ X, H_{\delta, p} \right]
$$
Note the definition of \eqref{initial vector fields}, we compute the right terms of the above equality separately, for any $f(t, x) \in C^\infty ([0, T], \mathscr{S}(\mathbb{R}^n))$,
\begin{align*}
\left[\partial_t, H_{\delta, p} \right] f &= \sum_{q = 0}^r (\partial_tI^q(h_\delta)(t)) X_{p, q}f\\
&= \delta t^{\delta - 1} X_{p, 0} f +\sum_{q = 1}^r I^{q-1}(h_\delta)(t) X_{p, q}f.
\end{align*}
and
\begin{align*}
\left[ X, H_{\delta, p} \right] f &= \sum_{q = 0}^r I^q(h_\delta)(t) \left[ X, X_{p, q} \right] f \\
&= -\sum_{q = 0}^r I^q(h_\delta)(t) X_{p, q+1} f,
\end{align*}
by the assumption ${\bf {(H_1)}}$, we have  $\left[ X, X_{p, r} \right] f =-X_{p, r+1} f= 0$, we get
$$
\left[ X, H_{\delta, p} \right] f = -\sum_{q = 1}^r I^{q-1}(h_\delta)(t) X_{p, q} f.
$$
Thus we have
$$
\left[ \partial_t + X, H_{\delta, p} \right] = \delta t^{\delta - 1} X_{p, 0} = \delta t^{\delta - 1} \partial_{x_p}.
$$
Assume now \eqref{2-3} holds true for $d \in \mathbb{N}^+$, then we study the case of $d + 1$, we have
\begin{equation*}
\begin{aligned}
\left[ \partial_t + X, H_{\delta, p}^{d + 1} \right] &= \left[ \partial_t + X, H_{\delta, p}^d H_{\delta, p} \right] \\
&= \left[ \partial_t + X, H_{\delta, p}^d \right] H_{\delta, p} + H_{\delta, p}^d \left[ \partial_t + X, H_{\delta, p} \right] \\
&= d \delta t^{\delta - 1} \partial_{x_p} H_{\delta, p}^{d - 1} H_{\delta, p} + H_{\delta, p}^d  \left( \delta t^{\delta - 1} \partial_{x_p} \right) \\
&= (d + 1) \delta t^{\delta - 1} \partial_{x_p} H_{\delta, p}^d,
\end{aligned}
\end{equation*}
here we use $ [\partial_{x_p}, H_{\delta, p}^{k}] =0$, and the following rule about the commutator,
$$
\left[ T_1, T_2 T_3 \right] = \left[ T_1, T_2 \right] T_3 + T_2 \left[ T_1, T_3 \right].
$$
where $T_1, T_2, T_3$ are operators.
\end{proof}

Throughout the paper, we assume that $(-k) ! = 1$ for any $k \in \mathbb{N}$. Now we study the estimates of $(a_{kj})$ and $(b_\ell)$.
\begin{lemma}\label{lemma2+}
Let $T > 0$ and $s \in \mathbb{R}$. Assume that $(a_{kj})$ and $(b_\ell)$ satisfying the assumption ${\bf (H_3)}$, then there exists a constant $\bar{B} > 0$ such that, for any $d \in \mathbb{N}$, $p, k, j \in \left\{ 1, \dots, m_0 \right\}, \ell \in \left\{ 0, \dots, m_0 \right\}$ and $0 \leq t \leq T$,
\begin{equation}\label{happy+}
\left\| H^d_{\delta, p} a_{kj} (t) \right\|_{H^{s_1} \left( \mathbb{R}^n \right)} \leq \bar{B}^{d + 1} (d - 2) !, \ \
\left\| H^d_{\delta, p} b_{\ell} (t) \right\|_{H^{s_1}\left( \mathbb{R}^n \right)} \leq \bar{B}^{d + 1} (d - 2) !, 
\end{equation}
 where $s_1 = |s| + \frac{n}{2} + 1$. 
\end{lemma}
\begin{proof}
For $p \in \left\{ 1, \dots, m_0 \right\}$, we have
\begin{equation}\label{H-p}
 H_{\delta, p} = t^\delta \Big( \sum_{q = 0}^r \frac{\Gamma(\delta + 1)}{\Gamma(\delta + 1 + q)} t^q X_{p, q} \Big) = t^\delta\Big( \sum^n_{i = 1} c_{pi}(t)\partial_{x_i}\Big),
\end{equation}
where $\{c_{pi}(t)\}$ are polynomials of $t$ with real constant coefficients, and we define
$$
M = \max\{|c_{pi}(t)| \big| \ 1\le p\le m_0, 1\le i\le n, 0 \le t \le T \} <\infty.
$$
We now consider the proof for the first term of \eqref{happy+}. For any $d \in \mathbb{N}$, we compute as follows 
\begin{equation*}
\begin{aligned}
\left\| H_{\delta, p}^d a_{kj}(t) \right\|_{H^{s_1}(\mathbb{R}^n)} &\leq C_6 \sum_{|\alpha| \leq s_1} \Big\| \partial^\alpha \Big( t^\delta\Big( \sum^n_{i=1}c_{pi}(t)\partial_{x_i}\Big) \Big)^d a_{kj}(t) \Big\|_{L^2(\mathbb{R}^n)} \\
&\leq C_6 (T^{\delta}M)^d \sum_{|\alpha| \leq s_1} \sum_{|\beta| = d} \frac{d!}{\beta!} \left\| \partial^{\alpha + \beta} a_{kj}(t) \right\|_{L^2(\mathbb{R}^n)}.
\end{aligned}
\end{equation*}
Since \eqref{1-2} and the following inequality, 
$$
(\alpha + \beta)! \leq 2^{|\alpha + \beta|} \alpha! \beta!, 
$$
we have 
\begin{align*}
&\left\| H_{\delta, p}^d a_{kj}(t) \right\|_{H^{s_1}(\mathbb{R}^n)} \leq C_6 (T^\delta M)^d \sum_{|\alpha| \leq s_1} \sum_{|\beta| = d} \frac{d!}{\beta!} (4B^2 + 1)^{s_1 + d + 1} \alpha! \beta! \\
&\qquad \leq C_6 (T^\delta M)^d \sum_{|\alpha| \leq s_1} \sum_{|\beta| = d} (d - 2)! d^2 (4B^2 + 1)^{s_1 + d + 1} (s_1!)^n \\
&\qquad \leq C_6 (T^\delta M)^d (n^{s_1 + 1} n^d 8^d) (s_1!)^n \big((4B^2 + 1)^{s_1 + 1}\big)^{d + 1} (d - 2)!.
\end{align*}
The proof for the second term of \eqref{happy+} is the same. 
\end{proof}

The same computation also gives the following estimate: for any $p, k, j \in \{ 1, \cdots, m_0 \}$ and $d \in \mathbb{N}$, there exists
\begin{equation}\label{happy}
\left\| H^d_{\delta, p} \partial_{x_k} a_{kj} (t) \right\|_{H^{s_1} \left( \mathbb{R}^n \right)} \leq \bar{B}^{d + 1} (d - 2) !, \quad \forall \ 0 \leq t \leq T, 
\end{equation}
here $\bar{B}$ is the real number as stated in Lemma \ref{lemma2+}.

We remark that the estimates \eqref{happy+} and \eqref{happy} are both used to obtain the following two estimates of the second-order differential operators in the Sobolev space $H^s$ with $s \in \mathbb{R}$.
\begin{lemma}\label{lemma3+?}
Let $T > 0$ and $s \in \mathbb{R}$. Assume that $(a_{kj})$ and $(b_\ell)$ satisfy the assumptions ${\bf (H_2)}$ and ${\bf (H_3)}$. Then there exists a constant $\tilde{B} > 0$ such that, for any $u \in C^\infty([0, T], \mathscr{S}(\mathbb{R}^n))$, $d \in \mathbb{N}$ and $p, k, j \in \left\{ 1, \dots, m_0 \right\}$, 
\begin{equation*}%\label{another 2-10}
\begin{split}
&\sum_{k, j = 1}^{m_0} \left|\left( H_{\delta, p}^d \left( \left( \partial_{x_k} a_{kj} \right) \left( \partial_{x_j} u \right) \right), H_{\delta, p}^d u \right)_{H^s}\right|
+ \left|\left( H_{\delta, p}^d (Y u ), H_{\delta, p}^d u \right)_{H^s}\right| \\
& \leq \sum_{\ell = 0}^d \binom{d}{\ell} \tilde{B}^{\ell + 1} (\ell - 2)!
\Big(\sum^{m_0}_{j=1}\left\| H_{\delta, p}^{d - \ell}  \partial_{x_j} u(t) \right\|_{H^s}
+\left\| H_{\delta, p}^{d - \ell} u(t) \right\|_{H^s}\Big)
\left\|H_{\delta, p}^d u(t) \right\|_{H^s}.
\end{split}
\end{equation*}
\end{lemma}
\begin{proof} 
We first claim the following Leibniz formula
$$
H_{\delta, p}^d \left( \left( \partial_{x_k} a_{kj} \right) \left( \partial_{x_j} u \right) \right) = \sum_{\ell = 0}^d \binom{d}{\ell} \left( H_{\delta, p}^\ell \partial_{x_k} a_{kj} \right) \left( H_{\delta, p}^{d - \ell} \partial_{x_j} u \right). 
$$
Then by using the Cauchy-Schwarz inequality and \eqref{commu-2}, \eqref{happy}, we compute as follows 
\begin{align*}
&\quad \sum_{k, j = 1}^{m_0} \big| \left( H_{\delta, p}^d \left( \left( \partial_{x_k} a_{kj} \right) \left( \partial_{x_j} u \right) \right), H_{\delta, p}^d u \right)_{H^s} \big| \\
&\leq \sum_{k, j = 1}^{m_0} \sum_{\ell = 0}^d \binom{d}{\ell} \left\| \big( H_{\delta, p}^\ell \partial_{x_k} a_{kj} \big) \big( H_{\delta, p}^{d - \ell}\partial_{x_j} u \big) \right\|_{H^s} \left\| H_{\delta, p}^d u \right\|_{H^s} \\
&\leq C_1 \sum_{k, j = 1}^{m_0} \sum_{\ell = 0}^d \binom{d}{\ell} \left\| H_{\delta, p}^\ell \partial_{x_k} a_{kj} \right\|_{H^{s_1}} \left\| H_{\delta, p}^{d - \ell} \partial_{x_j} u \right\|_{H^s} \left\| H_{\delta, p}^d u \right\|_{H^s} \\
&\leq \sum_{\ell = 0}^d \binom{d}{\ell} (m_0 C_1 \bar{B})^{\ell + 1} (\ell - 2)!  \sum_{j = 1}^{m_0} \left\| H_{\delta, p}^{d - \ell} \partial_{x_j} u \right\|_{H^s} \left\| H_{\delta, p}^d u \right\|_{H^s}.
\end{align*}
The same computation applies to the estimate of $\big|\big( H_{\delta, p}^d (Y u ), H_{\delta, p}^d u \big)_{H^s}\big|$.  
\end{proof}

The other estimate of second-order differential operator is stated as follows.
\begin{lemma}\label{lemma3}
Let $T > 0$ and $s \in \mathbb{R}$. Assume that $(a_{kj})$ satisfies the assumptions ${\bf (H_2)}$ and ${\bf (H_3)}$. Then there exists a constant $\tilde{B} > 0$ such that, for any $u\in C^\infty([0, T], \mathscr{S}(\mathbb{R}^n))$, $d \in \mathbb{N}$ and $p, k, j \in \left\{ 1, \dots, m_0 \right\}$, 
\begin{equation}\label{2-10}
\begin{split}
	&-\sum_{k, j = 1}^{m_0} \left( \partial_{x_k} H_{\delta, p}^d \left( a_{kj} \partial_{x_j} u \right), H_{\delta, p}^d u \right)_{H^s} \ge \Lambda^{-1} \sum_{k = 1}^{m_0} \left\| \partial_{x_k} H_{\delta, p}^d u(t) \right\|^2_{H^s}  \\
	&\qquad - \sum_{\ell = 1}^d \binom{d}{\ell} \tilde{B}^{\ell + 1} (\ell - 2)! \sum_{j = 1}^{m_0} \left\| H_{\delta, p}^{d - \ell} \partial_{x_j} u(t) \right\|_{H^s}  \sum_{k = 1}^{m_0} \left\| \partial_{x_k} H_{\delta, p}^d u(t) \right\|_{H^s} \\
	&\qquad - \tilde{B} \Big( \sum_{k = 1}^{m_0} \left\| \partial_{x_k} H_{\delta, p}^{d} u(t) \right\|_{H^s} \Big) \left\| H_{\delta, p}^d u(t) \right\|_{H^s}.
\end{split}
\end{equation}
\end{lemma}
\begin{proof}
By using the Leibniz formula, we have
\begin{align*}
 -\sum_{k, j = 1}^{m_0} \big( \partial_{x_k} H^d_{\delta, p}& \left( a_{kj} \partial_{x_j} u \right),  H^d_{\delta, p} u \big)_{H^s} 
  = \sum_{k, j = 1}^{m_0} \left( H_{\delta, p}^d \left( a_{kj} \partial_{x_j} u \right), \partial_{x_k} H_{\delta, p}^d u \right)_{H^s} \\
&= \sum_{k, j = 1}^{m_0} \left( a_{kj} H_{\delta, p}^d \partial_{x_j} u, \partial_{x_k} H_{\delta, p}^d u \right)_{H^s}\\
 & + \sum_{k, j = 1}^{m_0} \sum_{\ell = 1}^d \binom{d}{\ell} \left( \left( H_{\delta, p}^\ell a_{kj} \right)  \left( H_{\delta, p}^{d - \ell} \partial_{x_j} u \right), \partial_{x_k} H_{\delta, p}^d u \right)_{H^s} \\
&= I_1 + I_2.
\end{align*}
For the term $I_1$, we compute as follows 
\begin{align*}
I_1 &= \sum_{k, j = 1}^{m_0} \left( a_{kj} \Lambda_x^s H_{\delta, p}^d \partial_{x_j} u, \Lambda_x^s \partial_{x_k} H_{\delta, p}^d u \right)_{L^2} \\
&\quad +  \sum_{k, j = 1}^{m_0} \left( \left[ \Lambda_x^s, a_{kj} \right] H_{\delta, p}^d \partial_{x_j} u, \Lambda_x^s \partial_{x_k} H_{\delta, p}^d u \right)_{L^2} \\
&= J_{1, 1} + J_{1, 2}.
\end{align*}
For $J_{1, 1}$, thanks to the assumption ${\bf (H_2)}$, we have 
$$
J_{1, 1} \ge \Lambda^{-1} \sum_{k = 1}^{m_0} \left\| H_{\delta, p}^d \partial_{x_k} u \right\|_{H^s}.
$$
For $J_{1, 2}$, by using the Cauchy-Schwarz inequality and \eqref{commu-1}, we have
$$
|J_{1, 2}| \leq \sum_{k, j = 1}^{m_0} C_1 \left\| a_{kj} \right\|_{H^{s_0}} \left\| H_{\delta, p}^d \partial_{x_j} u \right\|_{H^{s - 1}} \left\| \partial_{x_k} H_{\delta, p}^d u \right\|_{H^s}.
$$
Noting the assumption ${\bf{(H_3)}}$, we compute as follows, 
\begin{align*}
\left\| a_{kj} \right\|_{H^{s_0}} \leq C_6 \sum_{|\alpha| \leq s_0} \left\| \partial^\alpha a_{kj} \right\|_{L^2} \leq C_6 \sum_{|\alpha| \leq s_0} (B + 1)^{s_0 + 1} (s!)^n.
\end{align*}
Now take $C_7 \ge C_1 C_6 \big( n (B + 1) \big)^{s_0 + 1} (s!)^n$, we then have 
\begin{align*}
|J_{1, 2}| &\leq \sum_{k, j = 1}^{m_0} C_7 \left\| H_{\delta, p}^d \partial_{x_j} u \right\|_{H^{s - 1}} \left\| \partial_{x_k} H_{\delta, p}^d u \right\|_{H^s} \\
&\leq C_7 m_0 \left\| H_{\delta, p}^d u \right\|_{H^s} \sum_{k = 1}^{m_0} \left\| \partial_{x_k} H_{\delta, p}^d u \right\|_{H^s}.
\end{align*}
For the term $I_2$, noting the estimate \eqref{happy+} and using \eqref{commu-2}, we compute as follows, 
\begin{align*}
|I_2| &\leq \sum_{k , j = 1}^{m_0} \sum_{\ell = 1}^d \binom{d}{\ell} \left\| \left( H_{\delta, p}^\ell a_{kj} \right)  \left( H_{\delta, p}^{d - \ell} \partial_{x_j} u \right) \right\|_{H^s} \left\| \partial_{x_k} H_{\delta, p}^d u \right\|_{H^s} \\
&\leq C_1 \sum_{k , j = 1}^{m_0} \sum_{\ell = 1}^d \binom{d}{\ell} \left\| H_{\delta, p}^\ell a_{kj} \right\|_{H^{s_1}} \left\| H_{\delta, p}^{d - \ell} \partial_{x_j} u \right\|_{H^s} \left\| \partial_{x_k} H_{\delta, p}^d u \right\|_{H^s} \\
&\leq \sum_{\ell = 1}^d \binom{d}{\ell} (C_1 \bar{B})^{\ell + 1} (\ell - 2)! \sum_{j = 1}^{m_0} \left\| H_{\delta, p}^{d - \ell} \partial_{x_j} u \right\|_{H^s} \sum_{k = 1}^{m_0} \left\| \partial_{x_k} H_{\delta, p}^d u \right\|_{H^s}.
\end{align*}
Combining all the estimates of $J_{1, 1}, J_{1, 2}$ and $I_2$, we then obtain the result \eqref{2-10}.
\end{proof}

For $g \in C^\infty([0, T]\times\mathbb{R}^n)$ satisfying the assumption ${\bf (H_3)}$, we set 
\begin{equation}\label{def-10}
g_\epsilon(t, x)=g(t, x) \, e^{-\epsilon |x|^{2}},\ \  0<\epsilon<1 ,
\end{equation}
then $g_\epsilon \in C^\infty([0, T], \mathscr{S}(\mathbb{R}^n))$ and we have,

\begin{lemma}\label{lemma4+}
Let $T > 0$ and $s \in \mathbb{R}$. Assume that $g_\epsilon$ is defined as \eqref{def-10} with $g$ satisfying the assumption ${\bf (H_3)}$. Then there exists a constant $\tilde{C} > 0$ such that, for any $d \in \mathbb{N}$ and $p \in \left\{ 1, \dots, m_0 \right\}$,
$$
\left\| H_{\delta, p}^d g_\epsilon(t) \right\|_{H^s(\mathbb{R}^n)} \leq  \tilde{C}^{d + 1} d!, \quad \forall \ 0 \leq t \leq T\ ,
$$
where $\tilde{C} > 0$ depends on $T$, but is independent of $0<\epsilon<1$. 
\end{lemma}
\begin{proof}
First we claim that there exists a constant $\bar{C} > 0$ {\em independent} of $\epsilon$ such that
$$
\left| \partial^\beta_x e^{-\epsilon |x|^2} \right| \leq \bar{C}^{|\beta| + 1} \beta!, \ \forall \beta \in \mathbb{N}^n,\ \forall\ x\in\mathbb{R}^n.
$$
Then for any $p \in \left\{ 1, \dots, m_0 \right\}$, noting \eqref{H-p}, we compute as follows, 
\begin{align*}
&\left\| H_{\delta, p}^d g_\epsilon(t) \right\|_{H^s} \leq C_8 \sum_{|\alpha| \leq [|s| ]} \Big\| \partial^\alpha \Big( t^\delta\Big( \sum^n_{j=1}c_{pj}(t)\partial_{x_j}\Big) \Big)^d (g \cdot e^{-\epsilon |x|^2}) \Big\|_{L^2} \\
&\qquad \leq C_8 (T^\delta M)^d \sum_{|\alpha| \leq [ |s| ]} \sum_{|\beta| = d} \frac{d!}{\beta!} \sum_{\gamma \leq \alpha + \beta} \binom{\alpha + \beta}{\gamma} \left\| \partial^{\alpha + \beta - \gamma} g \right\|_{L^2} \left\| \partial^{\gamma} e^{-\epsilon | x |^2} \right\|_{L^\infty} \\
&\qquad \leq C_8 (T^\delta M)^d d! \sum_{|\alpha| \leq [ |s| ]} \sum_{|\beta| = d} \sum_{\gamma \leq \alpha + \beta} 2^{|\alpha + \beta|} \alpha! \cdot C^{|\alpha + \beta - \gamma| + 1} \bar{C}^{|\gamma| + 1} \\
&\qquad \leq C_8 (T^\delta M [ |s|]^n)^{d} d! \sum_{|\alpha| \leq [|s| ]} \sum_{|\beta| = d} \sum_{\gamma \leq \alpha + \beta} \Big( \big(2(C + \bar{C})\big)^{[|s|] + 2} \Big)^{d + 1} \leq \tilde{C}^{d + 1} d!, 
\end{align*}
here $[\cdot ]$ stands for the symbol of rounding up.
\end{proof}

Another estimate is also needed in Section \ref{section5}.
\begin{lemma}\label{estimate up}
For any $f \in L^\infty([0, T], \mathscr{S}(\mathbb{R}^n))$, $d \in \mathbb{N}^+$ and $p \in \{ 1, \dots, m_0 \}$, we have 
$$
\left\| H_{\delta, p}^d f(t) \right\|_{H^s}^2 \bigg|_{t = 0} = 0.
$$
\end{lemma}
Here, we just need to use \eqref{H-p}, and for $0\le t\le T$,
\begin{align*}
&\left\| H_{\delta, p}^d f(t) \right\|_{H^s} = \Big\| \Big( t^\delta\Big( \sum^n_{j=1}c_{pj}(t)\partial_{x_j}\Big) \Big)^d f(t) \Big\|_{H^s} \\
&\qquad \leq t^{\delta d} M^d \sum_{d_1 + \dots + d_n = d} \frac{d!}{d_1! \dots d_n!} \left\| \partial^{d_1}_{x_1}\cdots \partial^{d_n}_{x_n}  f(t) \right\|_{H^s}.
\end{align*}
Note that $f \in L^\infty ([0, T], \mathscr{S}(\mathbb{R}^n))$, then $
\left\| \partial^{d_1}_{x_1}\cdots \partial^{d_n}_{x_n}  f(t) \right\|_{H^s}$ is bounded for $t\in [0, T]$, thus for $\delta>1$ and for any $d \in \mathbb{N}^+$, 
$$
\left\| H_{\delta, p}^d f(t) \right\|_{H^s}^2 \bigg|_{t = 0} =\lim_{t \to 0+}\left\| H_{\delta, p}^d f(t) \right\|_{H^s}^2 = 0.
$$

\section{Estimation of directional derivations}\label{section5}
In this section, we will give \`a priori estimate of $H_{\delta, p}^d u(t)$ for positive time and $d\in\mathbb{N}$.
\begin{thm}\label{thm3}
Let $T > 0, s\in\mathbb{R}, u_0\in H^s(\mathbb{R}^n)$. Under the assumptions of Theorem \ref{thm1}, assume that $u$ is the weak solution of the Cauchy problem \eqref{1-1}. Then $u(t)$ is smooth when $t >0$, and there exists a constant $A > 0$ such that, for any $d \in \mathbb{N}$, $p \in \left\{ 1, \dots, m_0 \right\}$ and $0<t \le T$, 
$$
\left\| H^d_{\delta, p} u(t) \right\|^2_{H^s} + \Lambda^{-1} \sum_{i = 1}^{m_0} \int_0^t \left\| \partial_{x_i} H^d_{\delta, p} u(\tau) \right\|_{H^s}^2 d\tau \leq (A^{d + 1} d!)^2,
$$
where A depends on $T, B, C, n, \Lambda, m_0, r$ and $\delta$ with $\delta>1$.
\end{thm}
Here the existence and uniqueness of the weak solution for the initial datum $u_0\in H^s(\mathbb{R}^n)$ is given by Corollary \ref{weak}.

\bigskip
\noindent{\bf Mollifier the initial data:} Let $u_0\in H^s(\mathbb{R}^n)$, we define
$$
u_{0, \epsilon}=(u_{0}*\varphi_{\epsilon})e^{-\epsilon |x|^{2}},\ \ \ 0<\epsilon<1,
$$
here $\varphi_{\epsilon}$ is defined as follows,
$$
\varphi_{\epsilon}(x)=\epsilon^{-n}\varphi(\epsilon^{-1} x),
$$
where $\varphi \in C^\infty_0(\mathbb{R}^n)$ satisfies $\varphi\ge 0$ and $\|\varphi\|_{L^1}=1$.
Then for any $m, \mu \in \mathbb{N}$, by using the Minkowski inequality and the Young inequality, we have
$$
\left\| u_{0, \epsilon} \right\|_{H^m_\mu}^2 \leq  C_{\epsilon, m, \mu} \left\| u_0 \right\|_{H^s}^2,
$$
which leads to $u_{0, \epsilon} \in \mathscr{S}(\mathbb{R}^n)$, we can also get
$$
\left\| u_{0, \epsilon} \right\|_{H^s} \leq \left\| u_0 \right\|_{H^s}.
$$

Now let $g_\epsilon$ be defined as in \eqref{def-10}, then by using Theorem \ref{thm2}, the following Cauchy problem
\begin{equation}\label{1-A}
\begin{cases}
\partial_t u_{\epsilon} + X u_{\epsilon}+ Y u_{\epsilon} + \displaystyle \sum_{k, j = 1}^{m_0} \left( \partial_{x_k} a_{ij} \right) \left( \partial_{x_j} u_{\epsilon} \right) - \displaystyle \sum_{k, j = 1}^{m_0} \partial_{x_k} (a_{ij} \partial_{x_j} u_{\epsilon}) = g_\epsilon\ , \\
u_{\epsilon} \big{|}_{t = 0} = u_{0, \epsilon},
\end{cases}
\end{equation}
admits a unique solution $u_{\epsilon}\in C^{\infty}([0, T]; \mathscr{S}(\mathbb{R}^n))$. From Theorem \ref{thm2}, the solution of the Cauchy problem \eqref{1-A} has the following estimate for $0 \leq t \leq T$,
\begin{equation}\label{A0}
\left\| u_\epsilon(t) \right\|^2_{H^s(\mathbb{R}^n)} + \Lambda^{-1} \sum_{k = 1}^{m_0} \int_0^t \left\| \partial_{x_k} u_\epsilon(\tau) \right\|^2_{H^s(\mathbb{R}^n)} d\tau \leq A_0^2.
\end{equation}
with
$$
A_0=\sqrt{C_T\big( \left\| u_0 \right\|^2_{H^s(\mathbb{R}^n)} +  \left\| g \right\|^2_{L^1 \left( [0, T], H^s(\mathbb{R}^n) \right)}\big)},
$$
which is independent of $\epsilon$.

The remainder of the proof of Theorem \ref{thm3} requires the following proposition.

\begin{prop}\label{prop4.1}
Let $T > 0, s\in\mathbb{R}, u_0\in H^s(\mathbb{R}^n)$. Under the assumptions of Theorem \ref{thm3}, assume $u_{\epsilon}$ is the solution of the Cauchy problem \eqref{1-A}. Then for any $d\in\mathbb{N}$, $p \in \left\{ 1, \dots, m_0 \right\}$ and $0 \leq t \leq T$, 
$$
\left\| H^d_{\delta, p} u_{\epsilon}(t) \right\|^2_{H^s(\mathbb{R}^n)} + \Lambda^{-1} \sum_{k = 1}^{m_0} \int_0^t \left\|  \partial_{x_k} H^d_{\delta, p} u_{\epsilon}(\tau) \right\|^2_{H^s(\mathbb{R}^n)} d\tau \leq \left( A^{d + 1} d! \right)^2,
$$
with $A >0$ independent of $ \epsilon$ and $d$.
\end{prop}

Since $A$ is independent of $\epsilon$, by using the compactness and uniqueness of the solution, we have then proved Theorem \ref{thm3}.

\begin{proof}[Proof of Proposition \ref{prop4.1} ]
For $d=0$, it has been proved as \eqref{A0}. Assume that $d \ge 1$, and for all $0 \leq m \leq d - 1$, $p \in \left\{ 1, \dots, m_0 \right\}, 0 \leq t \leq T$, we have the induction assumption:  
\begin{equation}\label{induction hypothesis}
     \left\| H^m_{\delta, p} u_{\epsilon}(t) \right\|^2_{H^s(\mathbb{R}^n)} + \Lambda^{-1} \sum_{k = 1}^{m_0} \int_0^t \left\| \partial_{x_k} H^m_{\delta, p} u_{\epsilon}(\tau) \right\|^2_{H^s(\mathbb{R}^n)} d\tau \leq \left( A^{m + 1} m!\right)^2.
\end{equation}
Now, we show that \eqref{induction hypothesis} holds true for $m=d$. Taking the derivative of the equation in \eqref{1-A}, we have
\begin{equation}\label{alone}
\begin{split}
	&\left( \partial_t + X \right) H_{\delta, p}^d u_\epsilon - \left[ \partial_t + X, H_{\delta, p}^d \right] u_\epsilon + H_{\delta, p}^d (Y u_\epsilon) \\
	&\quad + \sum_{k, j = 1}^{m_0} H_{\delta, p}^d \left( \left( \partial_{x_k} a_{kj} \right) \left( \partial_{x_j} u \right) \right) - \sum_{k, j = 1}^{m_0} H_{\delta, p}^d \partial_{x_k} \left( a_{kj} \partial_{x_j} u_\epsilon \right) = H_{\delta, p}^d g_\epsilon .
\end{split}
\end{equation}
Then by taking scalar product in $H^s$ with $H_{\delta, p}^d u_\epsilon$ in \eqref{alone}, we obtain
\begin{align*}
&\quad \frac{1}{2} \frac{d}{dt} \left\| H_{\delta, p}^d u_\epsilon \right\|_{H^s}^2 +\left( X H_{\delta, p}^d u_\epsilon, H_{\delta, p}^d u_\epsilon \right)_{H^s} - \left( \left[ \partial_t + X, H_{\delta, p}^d \right] u_\epsilon, H_{\delta, p}^d u_\epsilon \right)_{H^s} \\
&+\sum_{k, j = 1}^{m_0} \left( H_{\delta, p}^d \left( (\partial_{x_k}a_{kj}) (\partial_{x_j} u_\epsilon) \right), H_{\delta, p}^d u_\epsilon \right)_{H^s}
 - \sum_{k, j = 1}^{m_0} \left( H_{\delta, p}^d \partial_{x_k}\left( a_{kj} \partial_{x_j} u_\epsilon \right), H_{\delta, p}^d u_\epsilon \right)_{H^s} \\
&\qquad + \left(H_{\delta, p}^d (Y u_\epsilon), H_{\delta, p}^d u_\epsilon \right)_{H^s} = \left( H_{\delta, p}^d g_\epsilon ,  H_{\delta, p}^d u_\epsilon \right)_{H^s}.
\end{align*}
Since $u_\epsilon \in C^\infty([0, T], \mathscr{S}(\mathbb{R}^n))$, we have
\begin{equation}\label{OO}
\left|\left( X H_{\delta, p}^d u_\epsilon, H_{\delta, p}^d u_\epsilon \right)_{H^s}\right| \le C_{10} \left\| H_{\delta, p}^d u_\epsilon \right\|_{H^s}^2, 
\end{equation}
the proof of \eqref{OO} is almost identical to that of Lemma \ref{prelemma-1}. By using Lemma \ref{lemma1} and the Cauchy-Schwarz inequality,
\begin{equation}\label{5-0-20}
\left|\left( \left[ \partial_t + X, H_{\delta, p}^d \right] u_\epsilon, H_{\delta, p}^d u_\epsilon \right)_{H^s}\right| \leq  \left( d \delta t^{\delta - 1} \right)^2 \left\| \partial_{x_p} H_{\delta, p}^{d - 1} u_\epsilon \right\|_{H^s}^2 + \left\| H_{\delta, p}^d u_\epsilon \right\|_{H^s}^2.
\end{equation}
Using the Cauchy-Schwarz inequality again, 
\begin{equation}\label{5-0-2}
 \left| \left( H_{\delta, p}^d g_\epsilon, H_{\delta, p}^d u_\epsilon \right)_{H^s}\right| \leq  \left\| H_{\delta, p}^d g_\epsilon \right\|_{H^s}^2 + \left\| H_{\delta, p}^d u_\epsilon \right\|_{H^s}^2.
\end{equation}
Using Lemma \ref{lemma3+?}, Lemma \ref{lemma3} and above estimates \eqref{OO}, \eqref{5-0-20},
\eqref{5-0-2}, we have
\begin{align*}
&\quad \frac{1}{2} \frac{d}{dt} \left\| H_{\delta, p}^d u_\epsilon \right\|_{H^s}^2 + \Lambda^{-1} \sum_{j = 1}^{m_0} \left\| \partial_{x_j} H_{\delta, p}^d u_\epsilon \right\|_{H^s}^2 \\
&\leq \left\| H_{\delta, p}^d g_\epsilon \right\|_{H^s}^2 + (2 + C_{9}) \left\| H_{\delta, p}^d u_\epsilon \right\|_{H^s}^2+ (d\delta t^{\delta - 1})^2 \left\| \partial_{x_p} H_{\delta, p}^{d - 1} u_\epsilon \right\|_{H^s}^2\\
&\quad + \sum_{\ell = 1}^d \binom{d}{\ell} \tilde{B}^{\ell + 1} (\ell - 2)! \sum_{j = 1}^{m_0} \left\| H_{\delta, p}^{d - \ell} \partial_{x_j} u_\epsilon \right\|_{H^s} \sum_{k = 1}^{m_0} \left\| \partial_{x_k} H_{\delta, p}^d u_\epsilon \right\|_{H^s} \\
&\quad + \sum_{\ell = 1}^d \binom{d}{\ell} \tilde{B}^{\ell + 1} (\ell - 2)! \Big( \sum_{j = 1}^{m_0} \left\| H_{\delta, p}^{d - \ell} \partial_{x_j} u_\epsilon \right\|_{H^s} + \left\| H_{\delta, p}^{d - \ell} u_\epsilon \right\|_{H^s} \Big) \left\| H_{\delta, p}^d u_\epsilon \right\|_{H^s} \\
&\quad + \tilde{B} \Big( \sum_{j = 1}^{m_0} \left\| H_{\delta, p}^d \partial_{x_j} u_\epsilon \right\|_{H^s} + \left\| H_{\delta, p}^d u_\epsilon \right\|_{H^s} \Big) \left\| H_{\delta, p}^d u_\epsilon \right\|_{H^s} \\
&\quad + \tilde{B} \Big( \sum_{k = 1}^{m_0} \left\| \partial_{x_k} H_{\delta, p}^d u_\epsilon \right\|_{H^s} \Big) \left\| H_{\delta, p}^d u_\epsilon \right\|_{H^s}.
\end{align*}
Now we integrate the above inequality from 0 to t,
\begin{align*}
&\quad \left\| H_{\delta, p}^d u_\epsilon (t) \right\|_{H^s}^2 +  2\Lambda^{-1} \sum_{j = 1}^{m_0} \int_0^t \left\| \partial_{x_j} H_{\delta, p}^d u_\epsilon (s) \right\|_{H^s}^2 ds \\
&\leq 2 \int_0^t \left\| H_{\delta, p}^d g_\epsilon(\tau) \right\|_{H^s}^2 d\tau
+ C_{10} \int^t_0\left\| H_{\delta, p}^d u_\epsilon (\tau)\right\|_{H^s}^2d\tau
+ C_{T}' d^2 \int_0^t \left\| H_{\delta, p}^{d - 1}\partial_{x_p} u_\epsilon (\tau) \right\|_{H^s}^2 d\tau \\
&\quad + 2 \sum_{\ell = 1}^d \binom{d}{\ell} \tilde{B}^{\ell + 1} (\ell - 2)! \sum_{k, j = 1}^{m_0} \int_0^t \left\| H_{\delta, p}^{d - \ell} \partial_{x_j} u_\epsilon (\tau) \right\|_{H^s} \left\| \partial_{x_k} H_{\delta, p}^d u_\epsilon (\tau) \right\|_{H^s} d\tau \\
&\quad + 2 \sum_{\ell = 1}^d \binom{d}{\ell} \tilde{B}^{\ell + 1} (\ell - 2)! \sum_{j = 1}^{m_0} \int_0^t \left\| H_{\delta, p}^{d - \ell} \partial_{x_j} u_\epsilon (\tau) \right\|_{H^s} \left\| H_{\delta, p}^d u_\epsilon (\tau) \right\|_{H^s} d\tau \\
&\quad + 2 \sum_{\ell = 1}^d \binom{d}{\ell} \tilde{B}^{\ell + 1} (\ell - 2)! \int_0^t \left\| H_{\delta, p}^{d - \ell} u_\epsilon (\tau) \right\|_{H^s} \left\| H_{\delta, p}^d u_\epsilon (\tau) \right\|_{H^s} d\tau \\
&\quad + 4\tilde{B} \sum_{j = 1}^{m_0} \int_0^t \left\| H_{\delta, p}^d \partial_{x_j} u_\epsilon (\tau) \right\|_{H^s} \left\| H_{\delta, p}^d u_\epsilon (\tau) \right\|_{H^s} d\tau, 
\end{align*}
here we use $\left\| H_{\delta, p}^{d} u_\epsilon (t) \right\|_{H^s}^2 \bigg|_{t = 0} = 0$, i. e. Lemma \ref{estimate up}.
By using the Cauchy-Schwarz inequality, we have
\begin{align*}
&\quad \left\| H_{\delta, p}^d u_\epsilon (t) \right\|_{H^s}^2 + \Lambda^{-1} \sum_{i = 1}^{m_0} \int_0^t \left\| \partial_{x_i} H_{\delta, p}^d u_\epsilon (s) \right\|_{H^s}^2 ds \\
&\leq  2\int_0^t \left\| H_{\delta, p}^d g_\epsilon(\tau) \right\|_{H^s}^2 d\tau
+ C_{11} \int^t_0\left\| H_{\delta, p}^d u_\epsilon (\tau)\right\|_{H^s}^2d\tau
+ C_T' d^2 \int_0^t \Big\| H_{\delta, p}^{d - 1}\partial_{x_p} u_\epsilon (\tau) \Big\|_{H^s}^2 d\tau \\
&\quad\quad + \sum_{\ell = 1}^d \binom{d}{\ell} {B'}^{\ell + 1} (\ell - 2)!\Big(\sum_{j = 1}^{m_0} \int^t_0 \left\| H_{\delta, p}^{d - \ell} \partial_{x_j} u_\epsilon(\tau) \right\|^2_{H^s}d\tau\Big)^{\frac 12} \\
&\quad\quad + \sum_{\ell = 1}^d \binom{d}{\ell} {B'}^{\ell + 1} (\ell - 2)!\Big( \int^t_0 \left\| H_{\delta, p}^{d - \ell} u_\epsilon(\tau) \right\|^2_{H^s}d\tau\Big)^{\frac 12} .
\end{align*}
Using the induction hypothesis \eqref{induction hypothesis} and Lemma \ref{lemma4+}, we obtain
\begin{align*}
&\left\| H_{\delta, p}^d u_\epsilon (t) \right\|_{H^s}^2 + \Lambda^{-1} \sum_{j = 1}^{m_0} \int_0^t \left\| \partial_{x_j} H_{\delta, p}^d u_\epsilon (s) \right\|_{H^s}^2 ds \\
&\qquad  \leq C_{12} \left( A^d d! \right)^2 + C_{13} \int_0^t \left\| H_{\delta, p}^d u_\epsilon (s) \right\|_{H^s}^2 ds\ .
\end{align*}
By using the Gronwall's inequality, we have
$$
\left\| H_{\delta, p}^d u_\epsilon (t) \right\|_{H^s}^2 \leq C_{12} \left[ 1 + C_{13} T e^{C_{13} T} \right] \left( A^d d! \right)^2,
$$
which leads to
$$
\left\| H_{\delta, p}^d u_\epsilon (t) \right\|_{H^s}^2 + \Lambda^{-1} \sum_{j = 1}^{m_0} \int_0^t \left\| \partial_{x_j} H_{\delta, p}^d u_\epsilon (s) \right\|_{H^s}^2 ds \leq C_{14} \left( A^d d! \right)^2.
$$
Now we take $A \ge \max \left\{ A_0, \sqrt{C_{14}} \right\}$ to obtain
$$
\left\| H^d_{\delta, p} u_{\epsilon}(t) \right\|^2_{H^s(\mathbb{R}^n)} + \Lambda^{-1}  \sum_{j = 1}^{m_0} \int_0^t \left\|  \partial_{x_j} H^d_{\delta, p} u_{\epsilon}(s) \right\|^2_{H^s(\mathbb{R}^n)} ds \leq \left( A^{d + 1} d! \right)^2.
$$
Thus, \eqref{induction hypothesis} holds true for $m = d$. And $A>0$ is independent of $\epsilon$ and $d$.
\end{proof}

\section{Analyticity of the solution}\label{section6}
In this section, we aim to prove the analyticity of the spatial variables $x\in\mathbb{R}^n$ for the solution of the Cauchy problem \eqref{1-1} when $t>0$. Recall that we have proved in Theorem \ref{thm3}, for  $\delta > 1$, there exists $A > 0$ such that, for any $d \in \mathbb{N}$,
\begin{equation}\label{5.1}
\left\| H^d_{\delta, p} u(t) \right\|_{H^s(\mathbb{R}^n)}\le A^{d+1}_\delta d!,\quad p=1, \cdots, m_0\ ,
\end{equation}
here we emphasize in particular that $A_\delta$ is {\em dependent} on $\delta$. We  will then use the above estimate \eqref{5.1} to prove the following proposition.
\begin{prop}\label{prop estimate1}
Let $u$ be the solution of the Cauchy problem in Theorem \ref{thm3} and $\delta > 1$. Then there exists $\tilde {A} > 0$ such that, for any $d \in \mathbb{N}$,   $p \in \left\{ 1, \dots, m_0 \right\}$ and $t\in ]0, T]$,
\begin{equation}\label{prior1}
t^{(\delta + r + \ell) d} \left\| X_{p, \ell}^{d} u(t) \right\|_{H^s(\mathbb{R}^n)} \leq \tilde{A}^{d + 1} d !, \ \ \ \ell = 0, 1, \dots, r.
\end{equation}
\end{prop}
We first try to represent $X_{p, \ell}$ as a combination of $H_{\delta+j, p}, j=0, \cdots, r$ with coefficients as polynomials of $t$.

Setting
\begin{equation}\label{uselater1}
H_{\delta, p}^{(0)} =(\delta + 1) \dots (\delta + r)H_{\delta, p}= \frac{\Gamma(\delta + 1 + r)}{\Gamma(\delta + 1)} H_{\delta, p},
\end{equation}
and
\begin{equation}\label{uselater-2}
H_{\delta, p}^{(1)} = (\delta + r + 1) t H_{\delta, p}^{(0)} - (\delta  + 1) H_{\delta + 1, p}^{(0)},
\end{equation}
then, by iteration, for any integer $k \in \mathbb{N}$,
\begin{equation}\label{iterative format}
H_{\delta, p}^{(k + 1)} = (\delta + r + 1 + k) t H_{\delta, p}^{(k)} - (\delta + 2k + 1) H_{\delta + 1, p}^{(k)}.
\end{equation}
Since ${\bf X}_{q} = 0$ if $q > r$ by the assumption ${\bf {(H_1)}}$, we only need to pay attention to the case of $0 \leq k \leq r$.

Since the family of vector fields $\{ H_{\delta, p}^{(k)};\ 0 \leq k \leq r, 1\le p\le m_0\}$ are just linear combinations of $\{ H_{\delta, p}; 1\le p\le m_0\}$ with coefficients as polynomials of $t$, by using \eqref{5.1}, we have,
\begin{lemma}\label{lemma estimate2}
Let $u$ be the solution of the Cauchy problem in Theorem \ref{thm3}. Then there exists $A' > 0$ such that, for any $d \in \mathbb{N}$, $p \in \left\{ 1, \dots, m_0 \right\}$ and $t\in [0, T]$,
\begin{equation}\label{prior2}
\Big\| \Big( H_{\delta, p}^{(k)} \Big)^d u(t) \Big\|_{H^s(\mathbb{R}^n)} \leq (A')^{d + 1} d!, \ \ \ k = 0, 1, \dots, r ,
\end{equation}
where $A'$ depends on $A_\delta, A_{\delta+1}, \cdots, A_{\delta+r}$ give in \eqref{5.1}.
\end{lemma}
The following two lemmas together provide an explicit relation between $\{H_{\delta, p}^{(k)}\}$ and $\{X_{p, \ell}\}$,
\begin{lemma}\label{formHk}
For $\delta > 1$, $p \in \left\{ 1, \dots, m_0 \right\}$, and $0 \leq k \leq r$, we have
\begin{equation}\label{general}
H_{\delta, p}^{(k)} = \sum_{q = k}^{r} \frac{q!}{(q - k)!} \frac{\Gamma (\delta + r + 1 + k)}{\Gamma (\delta + q + 1 + k)}  t^{\delta + k + q} X_{p, q}.
\end{equation}
\end{lemma}
Inversely, we can represent $\left\{ X_{p, \ell};\ 0 \leq \ell \leq r \right\}$ by using $\{ H_{\delta, p}^{(k)};\ 0 \leq k \leq r\}$ as follows.
\begin{lemma}\label{Xpell-relationship}
For any $p \in \left\{ 1, \dots, m_0 \right\}$, we have
\begin{equation}\label{pro5-1}
t^{\delta+2r} X_{p, r}=\frac{1}{r!} H^{(r)}_{\delta, p},
\end{equation}
and for $0\le \ell \le r-1$,
\begin{equation}\label{pro5-2}
\begin{aligned}
t^{\delta + r + \ell} X_{p, \ell} =& \frac{\Gamma (\delta + r + 1 + \ell)}{\ell! \Gamma (\delta + 2r + 1)} \Big\{ H_{\delta + r - \ell, p}^{(\ell)}\\
 &- \sum_{q = \ell + 1}^r \frac{q!}{\left( q - \ell \right)!} \frac{\Gamma (\delta + 2r + 1)}{\Gamma (\delta + r + 1 + q)} t^{\delta + r + q} X_{p, q} \Big\}.
\end{aligned}
\end{equation}
Here the right hand side is well defined iteratively for the second index of $X_{p, q}$.
\end{lemma}
We will provide detailed proofs of the above three lemmas in Section \ref{S-6}. We can now complete the proof of Proposition \ref{prop estimate1}.

\begin{proof}[{\bf Proof of the Proposition \ref{prop estimate1}}]

We prove,  by induction on index $\ell$ from $\ell=r$ decrease to $\ell=0$, the following estimates
\begin{equation}\label{ano}
t^{(\delta + r + \ell)d} \left\| X_{p, \ell}^d\ u(t) \right\|_{H^s} \leq \tilde{A}_{\ell}^{d + 1} d!, \ \ \ell = 0, 1, \dots, r.
\end{equation}
For the case of $\ell = r$, by using \eqref{prior2} and \eqref{pro5-1}, we have,
$$
t^{(\delta + 2r)d} \left\| X_{p, r}^d u(t) \right\|_{H^s} = \Big( \frac{1}{r!} \Big)^d \Big\| \Big( H_{\delta, p}^{(r)} \Big)^d u(t) \Big\|_{H^s} \leq \Big\| \Big( H_{\delta, p}^{(r)} \Big)^d u(t) \Big\|_{H^s} \leq (A')^{d + 1} d!,
$$
we take $\tilde{A}_r = A'$ to get the result. Assume $\ell \le r - 1$, and for all $m$ with  $\ell + 1 \le m \le r$, \eqref{ano} is true, that means we have
\begin{equation}\label{preass-1}
t^{(\delta + r + m)d} \left\| X_{p, m}^d u(t) \right\|_{H^s} \leq \tilde{A}_{m}^{d + 1} d!.
\end{equation}
We now prove that \eqref{preass-1} holds true for $m = \ell$.

By using \eqref{pro5-2}, we have
\begin{align*}
&\quad t^{(\delta + r + \ell)d} \left\| X_{p, \ell}^d u(t) \right\|_{H^s} \\
&\leq \big\| \big( H_{\delta + r - \ell, p}^{(\ell)} - \sum_{q = \ell + 1}^r \frac{q!}{(q -\ell)!} \frac{\Gamma (\delta + 2r + 1)}{\Gamma (\delta + r + 1 + q)} t^{\delta + r + q} X_{p, q} \big)^d u(t) \big\|_{H^s} \\
&\leq 2^{rd} \Big( \Big\| \left( H_{\delta + r - \ell, p}^{(\ell)} \right)^d u(t) \Big\|_{H^s} \\
&\quad + \left( r !\ \Gamma(\delta + 2r + 1) \right)^d \ \sum_{q = \ell + 1}^r t^{(\delta + r + q)d} \left\| X_{p, q}^d u(t) \right\|_{H^s} \Big).
\end{align*}
Using \eqref{prior2} and the induction assumption \eqref{preass-1}, we have
\begin{align*}
&\quad t^{(\delta + r + \ell)d} \left\| X_{p, \ell}^d u(t) \right\|_{H^s} %\\ &
\leq 2^{rd} \Big( (A')^{d + 1} d!+ \left( r !\ \Gamma(\delta + 2r + 1) \right)^d \ \sum_{q = \ell + 1}^r  \tilde{A}_{q}^{d + 1} d! \Big) \\
&\qquad\qquad \leq \tilde{A}_\ell^{d + 1} d\ !,
\end{align*}
with $\tilde{A}_\ell=2^r (r+1)! \cdot \Gamma(\delta + 2r + 1) \cdot \max\{A', \tilde{A}_{q}; \ \ell+1\le q\le r\}$.
Now we take $\tilde{A} = \max \{ \tilde{A}_0, \dots, \tilde{A}_r \}$ to get the result of Proposition \ref{prop estimate1}.
\end{proof}

Remark that here we use the fact that if the operators $T_1, T_2$ are defined by the Fourier multipliers $T_1(\xi)$ and $T_2(\xi)$ , then
\begin{align*}
\| (T_1 + T_2)^d {u}\|_{H^s(\mathbb{R}^n)}&=
\| \left< \xi \right>^s (T_1(\xi) + T_2(\xi))^d \hat{u}\|_{L^2(\mathbb{R}^n)}\\
&\le 2^d\| \left< \xi \right>^s (|T_1(\xi)|^d + |T_2(\xi)|^d) \hat{u} \|_{L^2(\mathbb{R}^n)}\\
&\le 2^{d} (\| T_1^d {u}\|_{H^s(\mathbb{R}^n)}+\| T_2^d {u}\|_{H^s(\mathbb{R}^n)}).
\end{align*}

\bigskip
\noindent{\bf Analyticity of solution: }
Now we establish the analyticity of each spatial variables. Remark that the estimates \eqref{prior1} with $q=0$ imply the analyticity of the variables $x_1, \cdots, x_{m_0}$. The assumption ${\bf (H_1)}$ implies
\begin{equation}\label{spandef}
\partial_{x_j}  = \sum_{p = 1}^{m_0} \sum_{\ell = 0}^{r} c_{p\ell}^j X_{p, \ell}, \quad \forall \ j \in \{ m_0+1, \dots, n \},
\end{equation}
with real constant coefficients. And we take
\begin{equation}\label{largest-1}
K = 1+ \max\{ \big| c_{p\ell}^j \big|, 1 \leq p \leq m_0, \ 0 \leq \ell \leq r,\ m_0+1\le j\le n\}\ .
\end{equation}

By using Proposition \ref{prop estimate1}, we have the following result.
\begin{prop}\label{ana-all}
Let $u$ be the solution of the Cauchy problem in Theorem \ref{thm1}. Then there exists a constant $\bar{A} > 0$ such that, for $j \in \{ 1, \dots, n \}, \delta>1$, $d \in \mathbb{N}$ and $0 < t \leq T$, 
\begin{equation}\label{ana-former}
t^{(\delta + 2r) d} \left\| \partial_{x_j}^d u(t) \right\|_{H^s(\mathbb{R}^n)} \leq \bar{A}^{d + 1} d!.
\end{equation}
\end{prop}
\begin{proof}
For any $j \in \left\{ m_0 + 1, \dots, n \right\}$, since \eqref{spandef} and \eqref{largest-1}, we compute as follows,
\begin{align*}
t^{(\delta + 2r)d} \left\| \partial_{x_j}^d u(t) \right\|_{H^s} &= t^{(\delta + 2r)d} \Big\| \Big( \sum_{p = 1}^{m_0} \sum_{\ell = 0}^r c_{p\ell}^j X_{p, \ell} \Big)^d u(t) \Big\|_{H^s} \\
&\leq t^{(\delta + 2r)d} \ 2^{(m_0 + r + 1)d} \ \sum_{p = 1}^{m_0} \sum_{\ell = 0}^r \left\| \left( c_{p\ell}^j X_{p, \ell} \right)^d u(t) \right\|_{H^s} \\
&\leq t^{(\delta + 2r)d} \ 2^{(m_0 + r + 1)d} K^d \ \sum_{p = 1}^{m_0} \sum_{\ell = 0}^r \left\| X_{p, \ell}^d u(t) \right\|_{H^s}.
\end{align*}
If $0 < t \leq 1$ to get the supreme value, then
\begin{align*}
t^{(\delta + 2r)d} \left\| \partial_{x_j}^d u(t) \right\|_{H^s} &\leq 2^{(m_0 + r + 1)d} K^d \  \sum_{p = 1}^{m_0} \sum_{\ell = 0}^r t^{(\delta + r + \ell)d} \left\| X_{p, \ell}^d u(t) \right\|_{H^s} \\
&\leq 2^{(m_0 + r + 1)d} K^d\ \sum_{p = 1}^{m_0} \sum_{\ell = 0}^r \tilde{A}^{d + 1} d! \\
&\leq \bar{A}_1^{d + 1} d !,
\end{align*}
if we take $\bar{A}_1 = 2^{m_0 + r + 1} K m_0 (r + 1)$.

If $T > 1$ and $1 < t \leq T$ to get the supreme value, then
\begin{align*}
t^{(\delta + 2r)d} \left\| \partial_{x_j}^d u(t) \right\|_{H^s} &\leq t^{(r - \ell)d} \ 2^{(m_0 + r + 1)d} K^d \ \sum_{p = 1}^{m_0} \sum_{\ell = 0}^r t^{(\delta + r + \ell)d} \left\| X_{p, \ell}^d u(t) \right\|_{H^s} \\
&\leq (T + 1)^{rd} \ 2^{(m_0 + r + 1)d} K^d \ \sum_{p = 1}^{m_0} \sum_{\ell = 0}^r \tilde{A}^{d + 1} d! \\
&\leq \bar{A}_2^{d + 1} d !,
\end{align*}
if we take $\bar{A}_2 = (T + 1)^r 2^{m_0 + r + 1} K m_0 (r + 1)$. Thus we can obtain
$$
t^{(\delta + 2r)d} \left\| \partial_{x_j}^d u(t) \right\|_{H^s} \leq \bar{A}^{d + 1}d!, \quad j \in \left\{ m_0 + 1, \dots, n \right\},
$$
if we take $\bar{A} = 1 + \bar{A}_2$.

For the case of $j \in \left\{ 1, \dots, m_0 \right\}$, we have $\partial_{x_j} = X_{j, 0}$, thus by using \eqref{prior1}, we can easily obtain the estimate \eqref{ana-former}.
\end{proof}

We now give the proof of Theorem \ref{thm1}.
\noindent
\begin{proof}[\bf {The End of Proof of Theorem \ref{thm1}. }]
By using the Plancherel theorem, it's easy to verify the following inequality, 
$$
\left\| \partial^\alpha u(t) \right\|_{H^s(\mathbb{R}^n)} \leq \sum_{j = 1}^n \left\| \partial_{x_j}^{\left| \alpha \right|} u(t) \right\|_{H^s(\mathbb{R}^n)}, \ \ \forall \alpha\in\mathbb{N}^n. 
$$
Then by proposition \ref{ana-all}, we have
\begin{align*}
t^{(\delta + 2r) |\alpha|} \left\| \partial^\alpha u(t) \right\|_{H^s(\mathbb{R}^n)} &\leq t^{(\delta + 2r) |\alpha|} \sum_{j = 1}^n \left\| \partial_{x_j}^{|\alpha|} u(t) \right\|_{H^s(\mathbb{R}^n)} \leq L_1^{|\alpha| + 1} |\alpha|!.
\end{align*}
if we take $L_1 = n\bar{A}$. Since
$
|\alpha| ! \leq \left( 2^{n} \right)^{|\alpha| + 1} \alpha!, 
$
then we have
$$
\sup_{0 < t \leq T} t^{(\delta + 2r)|\alpha|} \left\| \partial^\alpha u(t) \right\|_{H^s(\mathbb{R}^n)} \leq L_1^{|\alpha| + 1} |\alpha|! \leq L^{|\alpha| + 1} \alpha!,
$$
where $L = 2^n L_1$,  which gives \eqref{analy-11}. Thus we have proved 
$$
u\in L^\infty(]0, T]; \mathcal{A}(\mathbb{R}^n)).
$$
Using the equation in \eqref{1-1}, and the fact that $\mathcal{A}(\mathbb{R}^n)$ is an algebra, we obtain
$$
u\in C^\infty(]0, T]; \mathcal{A}(\mathbb{R}^n)).
$$
\end{proof}

\section{Proofs of technic Lemmas}\label{S-6}
In this section, we give the proofs of the following technical Lemmas: Lemma \ref{lemma-0} , Lemma \ref{lemma estimate2}, 
Lemma \ref{formHk} and Lemma \ref{Xpell-relationship}.
\noindent
\begin{proof} [{\bf {Proof of Lemma \ref{lemma-0}. }}]
For fixed $s \in \mathbb{R}$, by using the properties of the Fourier transform, we compute directly as follows, 
\begin{align*}
\mathcal{F}\left(\left[ h, \Lambda_x^s \right] f\right)&=\mathcal{F}\left( h( \Lambda_x^s   f)\right)- 
\mathcal{F}\left(\Lambda_x^s (h f)\right)\\
&=\hat{h}*(\langle \cdot \rangle ^s \hat{f})-\langle \xi \rangle ^s\hat{h}*\hat{f} \\
&=\int_{\mathbb{R}^n}\hat{h}(\eta)(\langle \xi-\eta \rangle ^s -\langle \xi \rangle ^s )\hat{f}(\xi-\eta)d\eta.
\end{align*}
By using the first-order Taylor expansion to obtain the following result, 
\begin{align*}
\langle \xi-\eta \rangle ^s -\langle \xi \rangle ^s &= -\int^1_0\frac{d}{dt}(\langle (\xi-\eta)+t\eta \rangle ^s  )dt \\
&= - s \int^1_0\Big(\left(1+|(\xi-\eta)+t\eta|^2 \right)^{\frac s2-1}((\xi-\eta)+t\eta)\cdot\eta \Big) dt, 
\end{align*}
thus we have
\begin{equation}\label{taylor-1}
\left| \mathcal{F}\left(\left[ h, \Lambda_x^s \right] f\right) \right| \leq |s| \int_{\mathbb{R}^n} | \hat{h}(\eta) |\int^1_0\langle (\xi - \eta) + t\eta \rangle^{s - 1} |\eta| dt \cdot |\hat{f}(\xi-\eta)| d\eta. 
\end{equation}
By using the Peetre's inequality, for $0 \leq t \leq 1$, we have 
$$
\frac{\langle (\xi-\eta)+t\eta \rangle ^{s-1}}{\langle \xi-\eta \rangle ^{s-1}}\le 2^{|s-1|} \langle \eta\rangle ^{|s-1|}. 
$$
Combining \eqref{taylor-1}, we have 
\begin{align*}
\left|\mathcal{F}\left(\left[ h, \Lambda_x^s \right] f\right)\right|
&\le C_s\int_{\mathbb{R}^n} |\langle \eta\rangle ^{|s-1|+1}  \hat{h}(\eta)|\ 
|\langle \xi-\eta \rangle ^{s-1} \hat{f}(\xi-\eta)|d\eta .
\end{align*}
By using the Young's convolution inequality, we obtain 
$$
\|\left[ h, \Lambda_x^s \right] f\|_{L^2} =\|\mathcal{F}\left(\left[ h, \Lambda_x^s \right] f\right)\|_{L^2} \le C_s \| F_1*F_2\|_{L^2}\le C_s
 \|F_1\|_{L^1}\|F_2\|_{L^2}, 
$$
with 
$$
F_1(\eta)=\langle \eta \rangle^{|s-1|+1}\hat{h}(\eta), \ \ \  F_2(\eta)= \langle \eta \rangle^{s-1}\hat{f}(\eta).
$$
The estimate of $F_1$ needs the H\"older's inequality, 
$$
 \|F_1\|_{L^1}\le \| \langle \eta \rangle^{-\frac n2-1}\|_{L^2} \ \| \langle \eta \rangle^{s_0}\hat{h}(\eta)\|_{L^2}\le C_{1}\|h\|_{H^{s_0}}, 
$$
with $s_0=|s-1|+\frac n2+2$. Thus the proof of inequality \eqref{commu-1} is accomplished. 

As for the \eqref{commu-2}, by using the Peetre's inequality, we compute directly as follows, 
$$
\| h\, f\|_{H^s(\mathbb{R}^n)}=\|\langle \cdot \rangle^{s}(\hat{h} *\hat{f})\|_{L^2} \leq \|\langle \cdot \rangle^{|s|}\hat{h}\|_{L^1}\|\langle \cdot \rangle^{s}\hat{f}\|_{L^2}
\le C_{2} \| h\|_{H^{s_1}(\mathbb{R}^n)} \|  f\|_{H^s(\mathbb{R}^n)}, 
$$
with $s_1=|s|+\frac n2+1$.
\end{proof}

\noindent
\begin{proof} [{\bf {Proof of Lemma \ref{lemma estimate2}. }}]
We prove that, by induction on the index $k$, there exists ${A'}_k > 0$ such that, for any $d \in \mathbb{N}$,
$$
\left\| \left( H_{\delta, p}^{(k)} \right)^d u(t) \right\|_{H^s} \leq ({A'}_k)^{d + 1} d!, \ \ \ k = 0, 1, \dots, r.
$$
For the case of $k = 0$, note \eqref{uselater1}, then due to \eqref{5.1}, we have
\begin{equation*}
\begin{aligned}
\left\| \left( H_{\delta, p}^{(0)} \right)^d u(t) \right\|_{H^s} &= \left( (\delta + 1) \dots (\delta + r) \right)^d \left\|  H_{\delta, p}^d u(t) \right\|_{H^s}
\leq  ({A'}_0)^{d + 1} d!,
\end{aligned}
\end{equation*}
if we take ${A'}_0 = (\delta + 1) \dots (\delta + r) A$. Assume $k \ge 1$, the following inequality holds true for all $m$ satisfying $0 \leq m \leq k - 1$,
$$
\Big\| \Big( H_{\delta, p}^{(m)} \Big)^d u(t) \Big\|_{H^s(\mathbb{R}^n)} \leq ({A'}_m)^{d + 1} d!.
$$
Now for the case of $m = k$, by the definition of  $H_{\delta, p}^{(k)}$ in \eqref{iterative format}, we have
\begin{equation*}
\begin{aligned}
\big\| \big( H_{\delta, p}^{(k)} \big)^d u(t) \big\|_{H^s} &= \big\| \big( (\delta + r + k) t H_{\delta, p}^{(k - 1)} - (\delta + 2k - 1) H_{\delta + 1, p}^{(k - 1)} \big)^d u(t) \big\|_{H^s} \\
&\leq 2^d \big( \big\| \big( (\delta + r + k)t H_{\delta, p}^{(k - 1)} \big)^d u(t) \big\|_{H^s} \\
&\quad\ \  + \big\| \big( (\delta + 2k - 1) H_{\delta + 1, p}^{(k - 1)} \big)^d u(t) \big\|_{H^s} \big) \\
&\leq ({A'}_k)^{d + 1} d!.
\end{aligned}
\end{equation*}
if we take ${A'}_k = 2 (\delta + 2r) (T + 1) {A'}_{k-1}$. Then we take $A' = \max \{ {A'}_0, \dots, {A'}_r \}$ to obtain \eqref{prior2}.
\end{proof}

\noindent
\begin{proof} [{\bf {Proof of Lemma \ref{formHk}. }}]
For the case of $k = 0$, by using \eqref{uselater1} and \eqref{2-1}, we have
$$%\begin{equation}\label{case-0}
H_{\delta, p}^{(0)} = \frac{\Gamma(\delta + 1 + r)}{\Gamma(\delta + 1)} \Big( \sum_{q = 0}^r \frac{\Gamma(\delta + 1)}{\Gamma(\delta + 1 + q)} t^{\delta + q} X_{p, q} \Big) = \sum_{q = 0}^r \frac{\Gamma(\delta + 1 + r)}{\Gamma(\delta + 1 + q)} t^{\delta + q} X_{p, q}.
$$%\end{equation}
For the case of $k = 1$, since
\begin{align*}
(\delta + r + 1) t H_{\delta, p}^{(0)} &= (\delta + r + 1) t \Big( \sum_{q = 0}^r \frac{\Gamma(\delta + 1 + r)}{\Gamma(\delta + 1 + q)} t^{\delta + q} X_{p, q} \Big) \\
&= \sum_{q = 0}^r \frac{\Gamma(\delta + 2 + r)}{\Gamma(\delta + 1 + q)} t^{\delta + q + 1} X_{p, q},
\end{align*}
and
\begin{align*}
&(\delta  + 1) H_{\delta + 1, p}^{(0)} = \sum_{q = 0}^r \frac{(\delta + 1) \Gamma(\delta + 2 + r)}{\Gamma(\delta + 2 + q)} t^{\delta + q + 1} X_{p, q}
\end{align*}
thus by using \eqref{uselater-2} we have
$$%\begin{equation}\label{iterative-1}
H_{\delta, p}^{(1)} = \sum_{q = 0}^r \frac{q \cdot \Gamma(\delta + 2 + r)}{\Gamma(\delta + 2 + q)} t^{\delta + q + 1} X_{p, q} = \sum_{q = 1}^r \frac{q \cdot \Gamma(\delta + 2 + r)}{\Gamma(\delta + 2 + q)} t^{\delta + q + 1} X_{p, q}.
$$%\end{equation}

We then prove \eqref{general} by induction on the index $k$. The case of $k = 0$ is evident. Assume that $1 \leq k \leq r$, for all $0 \leq m \leq k - 1$, we have
\begin{equation}\label{assump-general}
H_{\delta, p}^{(m)} = \sum_{q = m}^{r} \frac{q!}{(q - m)!} \frac{\Gamma (\delta + r + 1 + m)}{\Gamma (\delta + q + 1 + m)}  t^{\delta + m + q} X_{p, q}.
\end{equation}
Now we show that \eqref{assump-general} holds true for $m = k$. By using \eqref{assump-general}, we can obtain
$$
H_{\delta, p}^{(k - 1)} = \sum_{q = k - 1}^r \frac{q!}{(q - k + 1)!} \frac{\Gamma(\delta + r + k)}{\Gamma(\delta + q + k)} t^{\delta + k - 1 + q} X_{p, q},
$$
and then
$$
H_{\delta + 1, p}^{(k - 1)} = \sum_{q = k - 1}^r \frac{q!}{(q - k + 1)!} \frac{\Gamma(\delta + 1 + r + k)}{\Gamma(\delta + 1 + q + k)} t^{\delta + k + q} X_{p, q}.
$$
Then by using the iterative format \eqref{iterative format}, we compute
\begin{align*}
H_{\delta, p}^{(k)} &=  (\delta + r + k) t H_{\delta, p}^{(k - 1)} - (\delta + 2k - 1) H_{\delta + 1, p}^{(k - 1)} \\
&= \sum_{q = k - 1}^r \frac{q!}{(q - k + 1)!} \big( \frac{(\delta + r + k)\Gamma (\delta + r + k)}{\Gamma(\delta + q + k)} \\
&\quad\ \  - \frac{(\delta + 2k - 1) \Gamma(\delta + 1 + r + k)}{\Gamma(\delta + 1 + q + k)} \big) t^{\delta + k + q} X_{p, q},
\end{align*}
by direct computation, we have
\begin{align*}
\frac{(\delta + r + k)\Gamma (\delta + r + k)}{\Gamma(\delta + q + k)} &- \frac{(\delta + 2k - 1) \Gamma(\delta + 1 + r + k)}{\Gamma(\delta + 1 + q + k)} \\
&= \frac{(q - k + 1) \Gamma(\delta + r + k + 1)}{\Gamma(\delta + q + k + 1)}.
\end{align*}
Thus
\begin{align*}
H_{\delta, p}^{(k)} &= \sum_{q = k - 1}^r \frac{q!}{(q - k + 1)!} \frac{(q - k + 1)\Gamma(\delta + r + k + 1)}{\Gamma(\delta + q + k + 1)} t^{\delta + k + q} X_{p, q} \\
&= \sum_{q = k}^r \frac{q!}{(q - k + 1)!} \frac{(q - k + 1)\Gamma(\delta + r + k + 1)}{\Gamma(\delta + q + k + 1)} t^{\delta + k + q} X_{p, q} \\
&= \sum_{j = k}^r \frac{q!}{(q - k)!} \frac{\Gamma(\delta + r + 1 + k)}{\Gamma (\delta + q + 1 + k)} t^{\delta + k + q} X_{p, q}.
\end{align*}
\end{proof}

\noindent
\begin{proof} [{\bf {Proof of Lemma \ref{Xpell-relationship}. }}]
For the case of $\ell = r$, we can obtain the following equality by \eqref{general},
$$
H_{\delta, p}^{(r)} = r! t^{\delta + 2r} X_{p, r}.
$$
For the case of $0 \leq \ell \leq r - 1$, by \eqref{general}, we have
\begin{align*}
H_{\delta + r - \ell, p}^{(\ell)} &= \sum_{q = \ell}^r \frac{q!}{(q - \ell)!} \frac{\Gamma (\delta + 2r + 1)}{\Gamma (\delta + r + 1 + q)} t^{\delta + r + q} X_{p, q} \\
&= \ell! \frac{\Gamma (\delta + 2r + 1)}{\Gamma (\delta + r + 1 + \ell)} t^{\delta + r + \ell} X_{p, \ell} \\
&\quad + \sum_{q = \ell + 1}^r \frac{q!}{(q - \ell)!} \frac{\Gamma (\delta + 2r + 1)}{\Gamma (\delta + r + 1 + q)} t^{\delta + r + q} X_{p, q},
\end{align*}
which lead to
\begin{align*}
t^{\delta + r + \ell} X_{p, \ell} =& \frac{\Gamma (\delta + r + 1 + \ell)}{\ell! \Gamma (\delta + 2r + 1)} \Big[ H_{\delta + r - \ell, p}^{(\ell)}\\
 &- \sum_{q = \ell + 1}^r \frac{q!}{\left( q - \ell \right)!} \frac{\Gamma (\delta + 2r + 1)}{\Gamma (\delta + r + 1 + q)} t^{\delta + r + q} X_{p, q} \Big].
\end{align*}
\end{proof}

\section{Some examples of ultra-parabolic equations}\label{S-7}
Now, we present some examples with different motivations and backgrounds of the equation in \eqref{1-1}.

\subsection*{Brownian motion with inertia}
In the aspect of the degenerate diffusion process, the following strongly degenerate equation
\begin{equation}\label{example-01}
\begin{cases}
&\partial_t u =\displaystyle \sum_{j = 1}^{n-1} x_j \cdot \nabla_{x_{j+1}} u+{\bf b}(t, x) \cdot \nabla_{x_{1}} u +  \sum_{k, j = 1}^m a_{kj} (t, x) \partial_{x_1^k}\partial_{x_1^j} u,\\
&u|_{t=0}=u_0,
\end{cases}
\end{equation}
is actually a class of Brownian motion processes with inertia proposed by Sonin \cite{Sonin-1},
where $x=(x_1, \cdots, x_n), x_j=(x^1_j, \cdots, x^m_j)\in\mathbb{R}^m, j=1, \cdots, n$ and
${\bf b}(t, x)=(b_1(t, x), \cdots, b_m(t, x))$. Brownian motion processes with inertia enjoy the property that the transition density is smooth even its diffusion matrix is degenerate \cite{ILL-2}\cite{ILL-1}.

We can now apply our Theorem \ref{thm1} to the Cauchy problem \eqref{example-01}. We have
$$
X= \sum_{j = 1}^{n-1} x_j \cdot \nabla_{x_{j+1}},\quad Y={\bf b}(t, x) \cdot \nabla_{x_{1}}.
$$
And
$$
{\bf X}_0 = \nabla_{x_{1}},\quad {\bf X}_1 =[{\bf X}_0, X]=\nabla_{x_{2}},\ \ \cdots,\ \
{\bf X}_{n-1} =[{\bf X}_{n-2}, X]=\nabla_{x_{n}}.
$$
Then the assumption ${\bf (H_1)}$ is satisfied, so that if the coefficients $(a_{k j}), (b_k)$ satisfy the assumption ${\bf (H_2), (H_3)}$, the solution $u(t, x)$ of Cauchy problem \eqref{example-01} is analytic for variables $x=(x_1, \cdots, x_n)\in\mathbb{R}^{n\times m}$ when $t>0$.

\subsection*{Fokker-Planck equation}

In the area of the kinetic theory, Landau equation models the binary grazing collisions between charged particles \cite{Villani-2}. The Cauchy problem for isotropic case of spatially inhomogeneous Landau equation is the following so-called Fokker-Planck equation \cite{Villani-1}
\begin{equation}\label{F-P}
\begin{cases}
\partial_t u + v \cdot \nabla_x u = \nabla_v \cdot (\nabla_v u + v u), \\
u|_{t=0}=u_0,
\end{cases}
\end{equation}
where $x, v \in \mathbb{R}^3$, and
$$
X =  v \cdot \nabla_x,\quad Y=-v \cdot \nabla_v-3,
$$
then
$$
{\bf X}_0 = \nabla_{v}, \ \ {\bf X}_1 = [{\bf X}_0, X] = \nabla_x,
$$
thus ${\bf (H_1)}$ is satisfied. The assumptions ${\bf (H_2)}$ and ${\bf (H_3)}$ are obvious. Then the solution $u(t, x)$ of the Cauchy problem \eqref{F-P} is analytic for the variables $(x, v) \in \mathbb{R}^3_x\times \mathbb{R}^3_v$ when $t > 0$.

Here we also mention that the auxiliary vector fields \eqref{2-1} can be used to obtain the sharp regularization effect for the non-cutoff Boltzmann equation as in \cite{ref11}.

\subsection*{Kolmogorov equation}

Kolmogorov \cite{Kol-1} in 1934 introduce the following equation
\begin{equation}\label{classical Kol}
\partial_t u + \sum_{j = 1}^n x_j \partial_{x_{n + j}} u - \sum_{j = 1}^n \partial_{x_j}^2 u = g,
\end{equation}
to describe the probability density of a system with $2n$ degree of freedom. In the work by Morimoto-Xu \cite{MX-1}, they prove that the related Cauchy problem of the \eqref{classical Kol} possess the analytical smoothing effect property.

A more general form of \eqref{classical Kol} is stated as follows,
\begin{equation}\label{high di}
\partial_t u + \sum_{j = 1}^n x_j \partial_{x_{n + j}} u - \sum_{j = 1}^n a_{j}(t, x) \partial_{x_j}^2 u = g,
\end{equation}
the Gevrey regularity results about the related Cauchy problem are obtained in \cite{o-3} and \cite{o-2}, the analytic regularity result about the two dimensional case is obtained by our former work \cite{u-1}. For the case of a more general $2n$-dimensional case, let
$$
X = \sum_{j = 1}^n x_j \partial_{x_{n + j}},
$$
then
$$
{\bf X}_0 = (\partial_{x_1}, \dots, \partial_{x_n}), \ \ {\bf X}_1 = [{\bf X}_0, X] = (\partial_{x_{n + 1}}, \dots, \partial_{x_{2n}}),
$$
which lead to the fact that the assumption ${\bf (H_1)}$ holds true. Thus once the functions $(a_j)$ are analytic, then the solution for the related Cauchy problem of \eqref{high di} is analytic with respect to the variable $x \in \mathbb{R}^{2n}$ for any positive time $t>0$.

\subsection*{Lanconelli-Polidoro type operators}

In \cite{ref1}, E. Lanconelli and S. Polidoro propose the following condition on the the coefficients of
 $X =\displaystyle \sum_{k, j = 1}^n b_{kj} x_k \partial_{x_j} $ of the equation in Cauchy problem \eqref{1-1}. The coefficient matrix $E = (b_{kj})_{n \times n}$ has the form, for $\ell \ge 1$,
\begin{equation}\label{matrix}
E = \begin{pmatrix}
\mathbb{O} & E_1 & \mathbb{O}  & \dots & \mathbb{O}  \\
\mathbb{O}  & \mathbb{O}  & E_2 & \dots & \mathbb{O}  \\
\cdot & \cdot & \cdot & \dots & \cdot \\
\cdot & \cdot & \cdot & \dots & \cdot \\
\cdot & \cdot & \cdot & \dots & \cdot \\
\mathbb{O}  & \mathbb{O}  & \mathbb{O}  & \dots & E_\ell \\
\mathbb{O}  & \mathbb{O}  & \mathbb{O}  & \dots & \mathbb{O}
\end{pmatrix}\, .
\end{equation}
{\bf (L-P)} Lanconelli-Polidoro condition:
$$
\begin{cases}
\mbox{$E_j$ is  $m_{j-1} \times m_j$ matrix with rank $m_j$, $j=1, \cdots, \ell$, and  }\\
\mbox{$m_0 \ge m_1 \ge \dots \ge m_\ell \ge 1$ with $m_0 + m_1 + \dots + m_\ell = n$.}
\end{cases}
$$
Right now, we verify that the assumption {\bf (L-P)} implies ${\bf (H_1)}$. Due to the structure of $(b_{kj})$ in \eqref{matrix}, we have
$$%\begin{equation}\label{initial def}
X = \sum_{k, j = 1}^{n} b_{kj} x_k \partial_{x_j}=\sum^\ell_{j=1}\langle x^{j-1} E_j, \nabla_j\rangle,,
$$%\end{equation}
where
$$
x^0=(x_1, \cdots, x_{m_0}),\ \ x^j=(x_{m_0+\cdots+m_{j-1}+1},\ \cdots,\  x_{m_0+\cdots+m_{j}})\in \mathbb{R}^{m_j},
$$
and
$$
\nabla_0=(\partial_{x_1}, \cdots, \partial_{x_{m_0}}),\ \
\nabla_j=(\partial_{x_{m_0+\cdots+m_{j-1}+1}},\ \cdots,\ \partial_{ x_{m_0+\cdots+m_{j}}}),\ \ \ j=1, \cdots, \ell.
$$
Then for any integer $1 \leq q \leq \ell$,
$$
{\bf X}_0 = \nabla_0 = (\partial_{x_1}, \cdots, \partial_{x_{m_0}}), \ \ {\bf X}_q = [{\bf X}_{q - 1}, X].
$$
Setting
$$%\begin{equation}\label{Eq1}
E^{(q)} = \prod_{i = 1}^q E_i, \ \ \ q = 1, \dots, \ell,
$$%\end{equation}
this is a $m_0\times m_q$ matrix. A direct linear algebra result is given as follows:
\begin{lemma}\label{lemma rank-1}
Let $E_1, \cdots, E_\ell$ in \eqref{matrix} satisfy the assumption {\bf (L-P)}, then
$$
\rank E^{(q)} = m_q,  \ \ \ q = 1, \dots, \ell.
$$
\end{lemma}

By using this lemma, we have the following result.
\begin{prop}\label{last-prop}
Let $E_1, \cdots, E_\ell$ in \eqref{matrix} satisfy the assumption {\bf (L-P)}, for any integer $1 \leq q \leq \ell$, then
\begin{equation}\label{forward}
{\bf X}^T_{q}= E^{(q)}\, \nabla^T_q,
\end{equation}
and there exists a $m_q\times m_0$ matrix $A_q$ with $\rank{A_q} = m_q$ such that
\begin{equation}\label{inverse}
\nabla^T_q=A_q {\bf X}_{q}^T.
\end{equation}
\end{prop}
\begin{proof}
We prove \eqref{forward} by induction on the index $q$. We now take any $f \in  \mathscr{S}(\mathbb{R}^n)$. For the case of $q = 1$, for any $k = 1, \dots, m_0$,
\begin{align*}
\left[ \partial_{x_k}, X \right] f &= \big[ \partial_{x_k}, \sum_{j = 1}^r \left< x^{j - 1} E_j, \nabla_j \right> \big] f \\
&= \big[ \partial_{x_k}, \left< x^0 E_1, \nabla_1 \right> \big] f \\
&= {\bf e}_k E_1 \nabla_1^T f,
\end{align*}
where ${\bf e}_k$ denotes the standard unit vector. Here we use an obvious fact that
$$
\bigg[ \partial_{x_k}, \sum_{j = 1, j \ne k}^r \left< x^{j - 1} E_j, \nabla_j \right> \bigg] = 0.
$$
Thus we can deduce that
\begin{equation*}
{\bf X}_1^T = \left[ \nabla_0, X \right] = E_1 \nabla_1^T = E^{(1)} \nabla_1^T.
\end{equation*}
Assume $1 \leq q \leq \ell$, \eqref{forward} holds true for all $1 \leq m \leq q - 1$,
\begin{equation}\label{forward-1}
{\bf X}^T_{m}= E^{(m)}\, \nabla^T_m.
\end{equation}
Now we show \eqref{forward-1} holds true for $m = q$, for any $k = 1, \dots, m_0$, we compute as follows,
\begin{align*}
[X_{k, q - 1}, X] f &= \big[ {\bf e}_k E^{(q - 1)} \cdot \nabla_{q - 1}^T, X \big] f \\
&= \big[ {\bf e}_k E^{(q - 1)} \cdot \nabla_{q - 1}^T, \sum_{j = 1}^e \left< x^{j - 1} E_j, \nabla_j \right> \big] f \\
&= \big[ {\bf e}_k E^{(q - 1)} \cdot \nabla_{q - 1}^T, \left< x^{q - 1} E_q, \nabla_q \right> \big] f \\
&= {\bf e}_k E^{(q - 1)} \cdot E_q \cdot \nabla_q^T f \\
&= {\bf e}_k E^{(q)} \cdot \nabla_q^T f.
\end{align*}
Here we use the fact that
$$
\bigg[ {\bf e}_k E^{(q - 1)} \cdot \nabla_{q - 1}^T, \sum_{j = 1, j \ne q}^r \left< x^{j - 1} E_j, \nabla_j \right> \bigg] = 0.
$$
Thus we have
$$
{\bf X}_q^T = [{\bf X}_{q - 1}, X]^T = \left( [X_{1, q - 1}, X], \dots, [X_{m_0, q - 1}, X] \right)^T = E^{(q)} \nabla_q^T.
$$
As for the validity of \eqref{inverse}, since $\rank E^{(q)} = m_q$ by lemma \ref{lemma rank-1}, we take $A_q = \left( \left( E^{(q)} \right)^T E^{(q)} \right)^{-1} \cdot \left( E^{(q)} \right)^T$. And a linear algebra result gives that $\rank A_q = m_q$.
\end{proof}

Thus by the result of the proposition \ref{last-prop}, the assumption ${\bf (H_1)}$ is satisfied, so that if the coefficients $(a_{kj})$ satisfy the assumption ${\bf (H_2), (H_3)}$, the solution of Cauchy problem \eqref{1-1} is analytic for variables $x=(x_1, \cdots, x_n)\in\mathbb{R}^{n}$ when $t>0$.

\bigskip\noindent {\bf Acknowledgements.}
This work was supported by the NSFC (No.12031006) and the Fundamental
Research Funds for the Central Universities of China.

\end{document}